\documentclass[preprint,12pt]{elsarticle}

\usepackage{amssymb}

\usepackage{amsfonts,amsmath,amscd,amsthm,color}

\newfont{\eufm}{eufm10 scaled\magstep1}

\newcommand{\cA}{\mathcal{A}}
\newcommand{\cC}{\mathcal{C}}
\newcommand{\cI}{\mathcal{I}}

\newcommand{\cD}{\mathcal{D}}
\newcommand{\cB}{\mathcal{B}}
\newcommand{\cV}{\mathcal{V}}
\newcommand{\cU}{\mathcal{U}}
\newcommand{\cM}{\mathcal{M}}
\newcommand{\cP}{\mathcal{P}}
\newcommand{\cF}{\mathcal{F}}

\newcommand{\cL}{\mathcal{L}}
\newcommand{\cG}{\mathcal{G}}

\newcommand{\cQ}{\mathcal{Q}}

\newcommand{\cK}{\mathcal{K}}

\newcommand{\bbN}{\mathbb{N}}
\newcommand{\bbZ}{\mathbb{Z}}

\newcommand{\bbQ}{\mathbb{Q}}
\newcommand{\bbI}{\mathbb{I}}

\newcommand{\bbP}{\mathbb{P}}
\newcommand{\bbF}{\mathbb{F}}

\newcommand{\bbH}{\mathbb{H}}
\newcommand{\bbK}{\mathbb{K}}

\newcommand{\bbD}{\mathbb{D}}

\def\para{\vspace{2mm}}
\def\id{{\rm ID}}

\def\dprim{\id{\rm prim}}
\def\dcont{\id{\rm cont}}
\def\dres{\partial{\rm Res}}

\def\dcres{\partial{\rm CRes}}

\def\dfres{\partial{\rm FRes}}

\def\ps{{\rm ps}}
\def\PS{{\rm PS}}

\def\rank{{\rm rank}}
\def\ord{{\rm ord}}

\def\sat{{\rm sat}}
\def\min{{\rm min}}
\def\max{{\rm max}}

\def\gcld{{\rm gcld}}
\def\c{{\rm c}}

\def\remm{{\rm rem}}

\def\sep{{\rm sep}}
\def\ldeg{{\rm ldeg}}

\newtheorem{thm}{Theorem}[section]
\newtheorem{lem}[thm]{Lemma}
\newtheorem{cor}[thm]{Corollary}
\newtheorem{prop}[thm]{Proposition}
\newtheorem{defi}[thm]{Definition}
\newtheorem{rem}[thm]{Remark}
\newtheorem{exs}[thm]{Examples}
\newtheorem{ex}[thm]{Example}

\journal{}

\begin{document}
\begin{frontmatter}
\title{Linear sparse differential resultant formulas}

\author{Sonia L. Rueda}
\address{Dpto. de Matem\' atica Aplicada, E.T.S. Arquitectura.\\
Universidad Polit\' ecnica de Madrid.\\
Avda. Juan de Herrera 4, 28040-Madrid, Spain.}
\ead{sonialuisa.rueda@upm.es}

\begin{abstract}
Let $\cP$ be a system of $n$ linear nonhomogeneous ordinary differential polynomials in a set $U$ of $n-1$ differential indeterminates.
Differential resultant formulas are presented to eliminate the differential indeterminates in $U$ from $\cP$. These formulas are determinants of coefficient matrices of appropriate sets of derivatives of the differential polynomials in $\cP$, or in a linear perturbation $\cP_{\varepsilon}$ of $\cP$.
In particular, the formula $\dfres(\cP)$ is the determinant of a matrix $\cM(\cP)$ having no zero columns if the system $\cP$ is "super essential".
As an application, if the system $\frak{P}$ is sparse generic, such formulas can be used to compute the differential resultant $\dres(\frak{P})$ introduced by Li, Gao and Yuan in \cite{LGY}.
\end{abstract}


\end{frontmatter}

%%%%%%%%%%%%%%%%%%%%%%%%%%%%%%%%%%%%%%%%%%%%%%%%%%%%%%%%%%%%%%%%%%%
\section{Introduction}
%%%%%%%%%%%%%%%%%%%%%%%%%%%%%%%%%%%%%%%%%%%%%%%%%%%%%%%%%%%%%%%%%%%

Elimination theory has proven to be a relevant tool in (differential) algebraic geometry (see \cite{Cox},\cite{Cox-2} and \cite{Bo}). Elimination techniques have been developed using Gr\"{o}bner bases, characteristic sets and (differential) resultants. The algebraic resultant has been broadly studied, regarding theory and computation, some significant references are \cite{GKZ}, \cite{CE}, \cite{St} and \cite{DA}. Meanwhile, its counterpart the differential resultant is at an initial state of development, a survey on this development can be found in the introductions of \cite{GLY} and \cite{Ru11}. Until very recently, the existing definitions of differential resultants for differential polynomials depended on the computation method \cite{CFproc}.
In the recent paper \cite{GLY}, a rigorous definition of the differential resultant $\dres(\frak{P})$, of a set $\frak{P}$ of $n$ nonhomogeneous generic ordinary differential polynomials in $n-1$ differential variables, has been presented: If the elimination ideal, of the differential ideal generated by $\frak{P}$, is $n-1$ dimensional then it equals the saturation ideal of a differential polynomial $\dres(\frak{P})$, the differential resultant of $\frak{P}$.  As in the algebraic case, the object that is naturally necessary for applications is the sparse differential resultant, and this was defined in \cite{LGY}, for a set of nonhomogeneous generic sparse ordinary differential polynomials.

The computation and applicability of sparse algebraic resultants attained great benefits from having close formulas for their representation \cite{DA}, \cite{SY}, \cite{EKK}. These formulas provide bounds for the degree of the elimination output and
ways of exploiting sparseness of the input polynomials on predicting the support of the elimination output. Namely, obtaining the Newton polytope of the resultant \cite{St}, whose support is a superset of the support of the resultant, reduces elimination to an interpolation problem in (numerical) linear algebra, \cite{SY}, \cite{Cue}, \cite{EKK}.

Sparse differential resultants can be computed with characteristic set methods for differential polynomials via symbolic computation algorithms \cite{BLM}, \cite{H}, \cite{GLY}, \cite{Ritt}.
The algorithms in \cite{H} and \cite{BLM} have been implemented in the Maple package {\rm diffalg}, \cite{BH} and in the {\rm BLAD libraries} \cite{BBLAD} respectively.
These methods do not have an elementary complexity bound \cite{GKOS} but,
a single exponential algorithm based on order and degree bounds of the sparse differential resultant has been recently proposed in \cite{LYG}. It would be useful to represent the sparse differential resultant as the quotient of two determinants, as done for the algebraic case in \cite{DA}.
As noted in \cite{LGY} and \cite{LYG}, having similar formulas in the differential case would improve the existing bounds for degree and order of the sparse differential resultant and therefore the existing algorithms for its computation. Matrix formulas would also contribute to the development of methods to predict the support of the sparse differential resultant, achieving similar benefits to the ones obtained in the algebraic case.
A matrix representation of the sparse differential resultant is important because it is the basis for efficient computation algorithms
and their study promises to have a grate contribution to the development and applicability of differential elimination techniques.

In the differential case, these so called Macaulay style formulas do not exist, even in the simplest situation.
The matrices used in the algebraic case to define the  Macaulay style formulas \cite{DA},
are coefficient matrices of sets of polynomials obtained by multiplying the original ones by appropriate sets of monomials, \cite{CE}. In the differential case, in addition, derivatives of the original polynomials should be considered. The differential resultant formula defined by Carr\`{a}-Ferro in \cite{CFproc}, is the algebraic resultant of Macaulay \cite{Mac}, of a set of derivatives of the ordinary differential polynomials in $\cP$. Already in the linear sparse generic case, these formulas vanish often, giving no information about the differential resultant  $\dres(\frak{P})$. A determinantal formula for $2$ generic
differential polynomials of arbitrary degree and order $1$ has been recently presented in \cite{ZYG}.

In this paper, determinantal formulas are provided for systems of $n$ linear nonhomogeneous ( non necessarily generic) differential polynomials $\cP$
 in $n-1$ differential indeterminates.
The linear case can be seen as a previous study to get ready to approach the nonlinear case. One can consider only the problem of taking the appropriate set of derivatives of the elements in $\cP$ and forget about the multiplication by sets of monomials for the moment.

Given $n$ differential polynomials, differential elimination is guaranteed of at most $n-1$ differential variables (see Section \ref{sec-linDPPEs}) but, if there were more, we may decide which ones to consider as part of the coefficients.
Take for instance the Lotka-Volterra equations
\begin{equation*}
\left\{
\begin{array}{c}
x'= \alpha x - \beta xy,\\
y'= \gamma y - \rho xy,
\end{array}
\right.
\end{equation*}
they can be looked at as a system given by two linear differential polynomials in the differential indeterminate $x$, with $\alpha$, $\beta$, $\gamma$ and $\rho$ algebraic constants,
\begin{align*}
&f_1(x)=(\beta y-\alpha)x+x'=a_1 x + a_2 x',\\
&f_2(x)=y'-\gamma y +\rho y x=b_0+b_1 x.
\end{align*}
Elimination of the $x$ differential variable can be achieved by the determinant of the coefficient matrix of $f_1(x)$, $f_2(x)$ and $f'_2(x)$,
\[
d ((y')^2-y y''+a y y'- a c y^2 -b y^2 y'+bc y^3).
\]

In \cite{RS}, the linear complete differential resultant $\dcres (\cP)$ of a set of linear differential polynomials $\cP$ was defined, as an improvement, in the linear case, of the differential resultant formula given by Carr\`{a}-Ferro. Still, $\dcres (\cP)$ is the determinant of a matrix having zero columns in many cases.
An implementation of the differential resultant formulas defined by Carr\`{a}-Ferro and the linear complete differential resultant defined in \cite{R11} is available at \cite{RDR}.

The linear differential polynomials in $\cP$ can be described via differential operators. We use appropriate bounds of the supports of those differential operators to decide on a convenient set $\ps(\cP)$ of derivatives of $\cP$, such that its coefficient matrix $\cM(\cP)$ is squared and, if $\cP$ is super essential (as defined in Section \ref{sec-diff res formulas}), it has no zero columns.
Obviously, $\det(\cM(\cP))$ could still be zero. In such case, we can always provide a linear perturbation $\cP_{\varepsilon}$ of $\cP$ such that $\det(\cM(\cP_{\varepsilon}))\neq 0$, as an adaptation of the perturbation methods described in \cite{R11} (for linear complete differential resultants) to the new formulas presented in this paper.
In the sparse generic case, we can guarantee that the linear sparse differential resultant $\dres(\frak{P})$ can always be computed via the determinant of the coefficient matrix $\cM(\frak{P}_{\varepsilon}^*)$ of a set $\ps(\frak{P}_{\varepsilon}^*)$, of derivatives of the elements in the perturbation of a super essential subsystem $\frak{P}^*$ of $\frak{P}$.
%As a consequence, an improvement is obtained of the existing bounds (see \cite{LGY} and \cite{LYG}) of the order
%of the differential resultant $\dres(\frak{P})$ of a sparse generic system $\frak{P}$ of linear nonhomogeneous ordinary differential polynomials.

Given a system of linear nonhomogeneous ordinary differential polynomials $\cP$, in Section \ref{sec-Preliminary}, we describe appropriate sets bounding the supports of the differential operators describing the polynomials in $\cP$.
Differential resultant formulas for $\cP$ are given in Section \ref{sec-diff res formulas}. In particular, the formula $\dfres(\cP)$ is defined, for the so called super essential (irredundant) systems, as the determinant of a matrix $\cM(\cP)$ with no zero columns. In Section \ref{sec-irredundant systems of dpls}, it is shown that every system $\cP$ contains a super essential subsystem $\cP^*$, which is unique if $\cP$ is differentially essential. Results on differential elimination for systems $\cP$ of linear differential polynomial parametric equations (linear DPPEs) are given in Section \ref{sec-linDPPEs}, including a perturbation
$\cP_{\varepsilon}$ such that if $\cP$ is super essential then $\dfres(\cP_{\varepsilon})\neq 0$. The methods in Section \ref{sec-linDPPEs} are used in Section \ref{sec-Computation of the sparse} to compute the differential resultant $\dres(\frak{P})$ of a linear nonhomogeneous generic sparse system $\frak{P}$ of
ordinary differential polynomials. As explained in Section \ref{sec-Computation of the sparse}, the differential resultant $\dres(\frak{P})$ exists only for differentially essential systems.

%%%%%%%%%%%%%%%%%%%%%%%%%%%%%%%%%%%%%%%%%%%%%%%%%%%%%%%%%%%%%%%%%%%
\section{Preliminary notions}\label{sec-Preliminary}
%%%%%%%%%%%%%%%%%%%%%%%%%%%%%%%%%%%%%%%%%%%%%%%%%%%%%%%%%%%%%%%%%%%
Let $\bbD$ be an ordinary differential domain with derivation
$\partial$. Let us consider the set $U=\{u_1,\ldots ,u_{n-1}\}$ of differential indeterminates over $\bbD$.
By $\bbN_0$ we mean the natural numbers including $0$.
For $k\in\bbN_0$, we denote by $u_{j,k}$ the $k$th derivative of $u_j$ and for $u_{j,0}$ we simply write $u_j$.
We denote by $\{U\}$ the set of derivatives of the elements of $U$,
$\{U\}=\{\partial^k u\mid u\in U,\; k\in \bbN_0\}$, and  by
$\bbD\{U\}$ the ring of differential polynomials in the differential
indeterminates $U$, which is a differential ring with derivation $\partial$,
\begin{displaymath}
\bbD\{U\}=\bbD[u_{j,k}\mid j=1,\ldots ,n-1,\; k\in \bbN_0 ].
\end{displaymath}
Given a subset $\cU\subset\{U\}$, we denote by $\bbD[\cU]$ the ring of polynomials in the indeterminates $\cU$.
Given $f\in\bbD\{U\}$ and $y\in U$, we denote by $\ord(f, y)$ the order of $f$ in the variable $y$. If $f$
does not have a term in $y$ then we define $\ord(f,y) = -1$. The order of $f$ equals $\max\{\ord(f,y)\mid y\in U\}$.

Let $\cK$ be a differential field of characteristic zero with derivation $\partial$ (e.g. $\cK=\bbQ(t)$, $\partial=\partial/\partial t$) and $C=\{c_1,\ldots ,c_n\}$ a set of differential indeterminates over $\cK$. The differential ring $\cK\{C\}$ is an example of differential domain.
By $\cK \langle C\rangle$ we denote the differential field extension of $\cK$ by $C$, the quotient field of $\cK\{C\}$.
The following rankings will be used throughout the paper (see \cite{Kol}, page 75):
\begin{itemize}
\item The order $u_1<\cdots <u_{n-1}$ induces an orderly ranking on $U$ (i.e. an order on $\{U\}$) as follows: $u_{i,j}<u_{k,l}$ $\Leftrightarrow$ $(j,i)<_{\rm lex}(l,k)$. We set $1<u_1$.

\item Let $(i,j),(k,l)\in \bbN_0^2$ be distinct. We write $(i,j)\prec (k,l)$ if $i>k$, or $i=k$ and $j<l$.
The order $c_n<\cdots <c_1$, induces a ranking on $C$, using the monomial order $\prec$: $c_{i,j}<c_{k,l}$ $\Leftrightarrow$ $(i,j)\prec (k,l)$.
\end{itemize}
We call $\frak{r}$ the ranking on $C\cup U$ that eliminates $U$ with respect to $C$, that is $\partial^k x< \partial^{k^\star} u$, for all $x\in C$, $u\in U$ and $k,k^\star\in\bbN_0$.
The previous are all classical concepts in differential algebra and references for them are \cite{Kol} and \cite{Ritt}.

Let $\cP:=\{f_1,\ldots ,f_n\}$ be a system of linear differential polynomials in $\bbD\{U\}$. We assume that:
\begin{enumerate}
\item[($\cP$1)] The order of $f_i$ is $o_i\geq 0$, $i=1,\ldots ,n$. So that no $f_i$ belongs to $\bbD$.
\item[($\cP$2)] $\cP$ contains $n$ distinct polynomials.
\item[($\cP$3)] $\cP$ is a nonhomogeneous system. There exist $a_i\in\bbD$ and $h_i$ homogeneous differential polynomial in $\bbD\{U\}$, such that $f_i(U)=a_i-h_i(U)$ and, for some $i\in\{1,\ldots ,n\}$, $a_i\neq 0$.
\end{enumerate}
We denote by $\bbD [\partial]$ the ring of differential operators with coefficients in $\bbD$.
There exist differential operators $\cL_{i,j}\in\bbD[\partial]$ such that
\[f_i=a_i+\sum_{j=1}^{n-1}\cL_{i,j} (u_j), a_i\in\bbD.\]
We denote by $|S|$ the number of elements of a set $S$.
We call the indeterminates $U$ a set of parameters.
The {\sf number of parameters} of $\cP$ equals
\begin{equation}\label{eq-nu}
\nu (\cP):=|\{j\in\{1,\ldots ,n-1\}\mid \cL_{i,j}\neq 0\mbox{ for some }i\in\{1,\ldots ,n\}\}|.
\end{equation}
In addition, we assume:
\begin{enumerate}
\item[($\cP$4)] $\nu (\cP)=n-1$.
\end{enumerate}

Let $[\cP]_{\bbD\{U\}}$ be the differential ideal generated by $\cP$ in $\bbD\{U\}$ (see \cite{Ritt}). Our goal is to define differential resultant formulas to compute elements of the elimination ideal
\[[\cP]_{\bbD\{U\}}\cap \bbD.\]
The assumption $\nu (\cP)=n-1$ guarantees $[\cP]_{\bbD\{U\}}\cap \bbD\neq \{0\}$ and allows the codimension one possibility, see Section \ref{sec-linDPPEs} and Example \ref{ex-motivation}(1). Nevertheless, in this paper we also deal with subsystems $\cP'$ of $\cP$ such that $\nu (\cP')\neq |\cP'|-1$ and the study of the consequences of the relation between $\nu(\cP')$ and $|\cP'|$ is central to this work.

\begin{exs}\label{ex-motivation}
\begin{enumerate}
\item Let us consider the system $\cP=\{f_1,f_2,f_3\}$ in $\bbD\{u_1,u_2\}$ with
\begin{equation}
\begin{array}{l}
f_1=a_1 + a_{1,1,0} u_1 + a_{1,1,1} u_{1,1} + a_{1,2,1} u_{2,1} + a_{1,2,2} u_{2,2},\\
f_2=a_2 + a_{2,2,2} u_{2,2} + a_{2,2,3} u_{2,3},\\
f_3=a_3 + a_{3,1,1} u_{1,1} + a_{3,2,1} u_{2,1} +a_{3,2,2} u_{2,2}.
\end{array}
\end{equation}
We assume that every coefficient $a_{i,j,k}$ is nonzero. The differential operators describing this system are
\begin{equation}
\begin{array}{ll}
\cL_{1,1}=a_{1,1,0} + a_{1,1,1}\partial, & \cL_{1,2}=a_{1,2,1}\partial + a_{1,2,2}\partial^2,\\
\cL_{2,1}=0,  & \cL_{2,2}=a_{2,2,2}\partial^2 + a_{2,2,3}\partial^3,\\
\cL_{3,1}=a_{3,1,1}\partial,  &\cL_{3,2}=a_{3,2,1}\partial +a_{3,2,2}\partial^2.
\end{array}
\end{equation}
Given $\cP'=\{f_1,f_2\}$, $\nu(\cP')=2=|\cP'|$ and it is easily seen that $[f_1,f_2]_{\bbD\{u_1,u_2\}}\cap \bbD=\{0\}$.
Thus we cannot use $\cP'$ to eliminate $u_1$ and its derivatives but we can eliminate $u_2$ and all its derivatives. Namely if
$\overline{\bbD}=\bbD\{u_1\}$ and if $a_1,a_2$ are differential indeterminates, by \cite{RS}, Algorithm 2 then there exist nonzero differential operators $\cL_1,\cL_2\in\overline{\bbD}[\partial]$ such that $\cL_1(a_1)+\cL_2(a_2)=\cL_1(f_1)+\cL_2(f_2)$ belongs to $[\cP']_{\overline{\bbD}\{u_2\}}\cap\overline{\bbD}\neq \{0\}$.

\item Let us consider a system $\cP=\{f_1=c_1+\cL_{1,1}(u_1),f_2=c_2+\cL_{2,1}(u_1),f_3=c_3+\cL_{3,1}(u_1)\}$ in $\bbD\{u_1\}$, with $\bbD=\cK\{c_1,c_2,c_3\}$ and each $f_i$ of nonzero order. Observe that
    \[\nu (\cP)=1<|\cP|-1=2.\]
    By \cite{RS}, Algorithm 2, there exist nonzero differential operators $\cL_1,\cL_2\in\bbD[\partial]$ such that
    \[R_1=\cL_1(c_1)+\cL_2(c_2)\in [f_1,f_2]_{\bbD\{u_1\}}\cap \bbD\]
    and nonzero $\cD_2,\cD_3\in\bbD[\partial]$ such that
    \[R_2=\cD_2(c_2)+\cD_3(c_3)\in [f_2,f_3]_{\bbD\{u_1\}}\cap \bbD.\]
    Thus $[\cP]_{\bbD\{u_1\}}\cap \bbD$ has codimension greater than one, it is generated at least by two differential polynomials.
\end{enumerate}
\end{exs}

Given a nonzero differential operator  $\cL=\sum_{k\in\bbN_0}\alpha_k\partial^k\in\bbD [\partial]$, let us denote the support of $\cL$ by
$\frak{S} (\cL)=\{k\in\bbN_0 \mid \alpha_k\neq 0\}$, and define
\begin{align*}
\ldeg (\cL):=\min\,\frak{S}(\cL), \deg (\cL):=\max\,\frak{S}(\cL).
\end{align*}
For $j=1,\ldots ,n-1$, we define the next positive integers, to construct convenient intervals bounding the supports of the differential operators $\cL_{i,j}$,
\begin{equation}\label{eq-gammaj}
\begin{array}{l}
\overline{\gamma}_j(\cP):=\min\{o_i-\deg(\cL_{i,j})\mid \cL_{i,j}\neq 0,\,i=1,\ldots ,n\},\\
\underline{\gamma}_j(\cP):=\min\{\ldeg(\cL_{i,j})\mid \cL_{i,j}\neq 0,\, i=1,\ldots ,n\},
\end{array}
\end{equation}
\[\gamma_j (\cP):=\underline{\gamma}_j(\cP)+\overline{\gamma}_j(\cP).\]
Given $j\in\{1,\ldots ,n-1\}$, observe that, for all $i$ such that $\cL_{i,j}\neq 0$ we have
\begin{equation}\label{eq-oigammaj}
\underline{\gamma}_j (\cP)\leq \ldeg(\cL_{i,j})\leq \deg (\cL_{i,j})\leq o_i-\overline{\gamma}_j (\cP).
\end{equation}
Therefore, for $\cL_{i,j}\neq 0$ the next set of lattice points contains $\frak{S}(\cL_{i,j})$ ,
\[I_{i,j}(\cP):=[\underline{\gamma}_j(\cP),o_i-\overline{\gamma}_j(\cP)]\cap\bbZ.\]
Finally, to explain the construction of Section \ref{sec-diff res formulas}, we will use the integer
\begin{equation}\label{eq-undgamma}
\gamma (\cP):=\sum_{j=1}^{n-1}\gamma_j (\cP).
\end{equation}

\begin{ex}\label{ex-gammas}
Let $\cP$ be as in Example \ref{ex-motivation}(1). We have
\[
\begin{array}{lll}
o_1=2, & \frak{S}(\cL_{1,1})=\{0,1\}, & \frak{S}(\cL_{1,2})=\{1,2\},\\
o_2=3, & \frak{S}(\cL_{2,1})=\emptyset, & \frak{S}(\cL_{2,2})=\{2,3\},\\
o_3=2, & \frak{S}(\cL_{3,1})=\{1\}, & \frak{S}(\cL_{3,2})=\{1,2\}.\\
\end{array}
\]
Thus
\[
\begin{array}{lll}
\underline{\gamma}_1(\cP)=0, & \overline{\gamma}_1(\cP)=1, & \gamma_1(\cP)=1,\\
\underline{\gamma}_2(\cP)=1, & \overline{\gamma}_2(\cP)=0, & \gamma_2 (\cP)=1,
\end{array}
\]
and $\gamma (\cP)=2$.
%\[
%\begin{array}{ll}
%I_{1,1}(\cP)=\{\}, & I_{1,2}(\cP)=,\\
%     & I_{2,2}(\cP)={2,3\},\\
%I_{3,1}(\cP)=\{1\}, & I_{3,2}(\cP)=\{1,2\}.\\
%\end{array}
%\]
\end{ex}

%%%%%%%%%%%%%%%%%%%%%%%%%%%%%%%%%%%%%%%%%%%%%%%%%%%%%%%%%%%%%%%%%%%
\section{Differential resultant formulas}\label{sec-diff res formulas}
%%%%%%%%%%%%%%%%%%%%%%%%%%%%%%%%%%%%%%%%%%%%%%%%%%%%%%%%%%%%%%%%%%%

Let us consider a subset $\PS$ of $\partial\cP:=\{\partial^k f_i\mid i=1,\ldots ,n, k\in\bbN_0\}$ and a set of differential indeterminates $\cU\subset\{U\}$ verifying:
\begin{itemize}
\item[(\ps1)] $\PS=\{\partial^k f_i\mid k\in [0,L_i]\cap\bbZ,\, L_i\in\bbN_0,\,i=1,\ldots ,n\}$,
\item[(\ps2)] $\PS\subset\bbD [\cU]$ and $|\cU|=|\PS|-1$.
\end{itemize}
Let $N:=\sum_{i=1}^n o_i$.
\begin{rem}\label{rem-psuR}
Particular cases of sets $\PS$ and $\cU$ verifying $(\ps1)$ and $(\ps2)$ were given in \cite{CFproc} and \cite{R11} (see also \cite{RS}).
\begin{enumerate}
\item In \cite{CFproc}, $L_i=N-o_i$ and  $\cU=\{u_{j,k}\mid k\in [0,N]\cap\bbZ,\, j=1,\ldots ,n-1\}$.

\item In \cite{R11}, Section 3, $L_i=N-o_i-\hat{\gamma}$, where $\hat{\gamma}:=\sum_{j=1}^{n-1}\hat{\gamma}_j$,
\[\hat{\gamma}_j:=\min\{\overline{\gamma}_j(\cP),\min\{o_i\mid\cL_{i,j}=0,\,i=1,\ldots ,n\}\},\]
and $\cU=\{u_{j,k}\mid k\in [0,N-\hat{\gamma}_j-\hat{\gamma}]\cap\bbZ,\, j=1,\ldots ,n-1\}$.
\end{enumerate}
Observe that both choices coincide if $\hat{\gamma}=0$.
\end{rem}

The coefficient matrix $\cM(\PS,\cU)$ of the differential polynomials in $\PS$ as polynomials in $\bbD[\cU]$ is a $|\PS|\times
|\PS|$ matrix.
\begin{defi}
Given $\PS$ and $\cU$ verifying $(\ps1)$ and $(\ps2)$, we call
\[\det(\cM(\PS,\cU))\]
a {\sf differential resultant formula} for $\cP$.
\end{defi}

It can be proved as in \cite{RS}, Proposition 16(1) that $\det(\cM(\PS,\cU))\in [\cP]_{\bbD\{U\}}\cap\bbD$. Therefore, if $\det(\cM(\PS,\cU))\neq 0$ then it serves for differential elimination of the variables $U$ from $\cP$. See Examples \ref{ex-dfres}.

\para

The differential resultant formulas for $\cP$ given in \cite{CFproc} and \cite{R11} are determinants of matrices with zero columns in many cases.
Let $\PS^h:=\{\partial^k h_i\mid \partial^k f_i\in\PS\}$, the set containing the homogeneous part of the polynomials in $\PS$.
The coefficient matrix
\begin{equation}\label{eq-Lps}
\cL (\PS,\cU)
\end{equation}
of $\PS^h$, as a set of polynomials in $\bbD[\cU]$, is a submatrix of $\cM(\PS,\cU)$ of size $|\PS|\times (|\PS|-1)$. We assumed that $\cP$ is a nonhomogeneous system, thus if $\cM(\PS,\cU)$ has zero columns, those are columns of $\cL (\PS,\cU)$.

\begin{rem}\label{rem-psuR2}
The differential resultant formula for $\cP$ given in \cite{R11} is called the linear complete differential resultant of $\cP$ and denoted $\dcres(\cP)$. With $\PS$ and $\cU$ as in Remark \ref{rem-psuR}(2), $\dcres(\cP)=\det(\cM(\PS,\cU))$. Observe that, if $\underline{\gamma}_j(\cP)\neq 0$ for some $j\in \{1,\ldots ,n-1\}$, then the columns of $\cL(\PS,\cU)$ indexed by $u_j,\ldots ,u_{j,\underline{\gamma}_j(\cP)-1}$ are zero.
If $\overline{\gamma}_j(\cP)>\hat{\gamma}_j$ for some $j\in \{1,\ldots ,n-1\}$, then the columns of $\cL(\PS,\cU)$ indexed by $u_{j,N-\overline{\gamma}_j(\cP)-\hat{\gamma}+1}\ldots ,u_{j,N-\hat{\gamma}_j-\hat{\gamma}}$ are zero.
\end{rem}

If $N-o_i-\gamma (\cP)\geq 0$, $i=1,\ldots ,n$, the sets of lattice points
$\bbI_i:=[0,N-o_i-\gamma (\cP)]\cap\bbZ$ are non empty. We define the set of differential polynomials
\begin{equation}\label{eq-ps}
\ps (\cP):=\{\partial^k f_i\mid k\in\bbI_i,\, i=1,\ldots ,n\},
\end{equation}
containing
\begin{equation}\label{eq-L}
L:=\sum_{i=1}^n (N-o_i-\gamma (\cP)+1)
\end{equation}
differential polynomials, in the set $\cV$ of $L-1$ differential indeterminates
\begin{equation}\label{eq-V}
\cV:=\{u_{j,k}\mid k\in [\underline{\gamma}_j(\cP),N-\overline{\gamma}_j(\cP)-\gamma (\cP)]\cap\bbZ,\, j=1, \ldots , n-1\}.
\end{equation}
Let us assume that $\ps (\cP)=\{P_1,\ldots ,P_L\}$.
For $i=1,\ldots ,n$ and $k\in \bbI_i$,
\begin{align*}
P_{l(i,k)}&:=\partial^{k}f_i,\\
l(i,k)&:=\sum_{h=1}^{i-1}(N-o_h-\gamma (\cP)+1)+N-o_i-\gamma (\cP)+1-k\in\{1,\ldots ,L\}.
\end{align*}
The matrix  $\cM(\cP):=\cM(\ps(\cP),\cV)$ is an $L\times L$ matrix. We assume that the $l$th row of $\cM(\cP)$, $l=1,\ldots ,L$ contains the coefficients of $P_l$ as a polynomial in $\bbD[\cV]$, and that the coefficients are written in
decreasing order with respect to the orderly ranking on $U$.

Thus, if $N-o_i-\gamma (\cP)\geq 0$, $i=1,\ldots ,n$, we can define a linear differential resultant formula for $\cP$, denoted by $\dfres (\cP)$, and equal to:
\begin{align}\label{eq-mac}
\dfres (\cP):=\det(\cM(\cP)).
\end{align}

In general, we cannot guarantee that the columns of $\cM(\cP)$ are nonzero, as the next example shows.
\begin{ex}\label{ex-notsupess}
Let $\cP=\{f_1,f_2,f_3\}$, with $o_1=5$, $o_2=1$, and $o_3=1$. Let
$f_1= a_1+\cL_{1,1}(u_1)$,$f_2=a_2+\cL_{2,2}(u_2)$, $f_3=a_3+\cL_{3,2}(u_2)$,
with $\frak{S}(\cL_{1,1})=\{1,5\}$ and $\frak{S}(\cL_{2,2})=\frak{S}(\cL_{3,2})=\{0,1\}$. Then
$\gamma(\cP)=\underline{\gamma}_1(\cP)=1$ and $N-o_1-\gamma(\cP)=1$, $N-o_2-\gamma(\cP)=N-o_3-\gamma(\cP)=5$. Therefore $\cM (\cP)$ can be defined but columns indexed by $u_{1,3}$ and $u_{1,4}$ are zero.
\end{ex}

%%%%%%%%%%%%%%%%%%%%%%%%%%%%%%%%%%%%%%%%%%%%%%%%%%%%%%%%%%%%%%%%%%%
%\subsection{Linear super essential systems}
%%%%%%%%%%%%%%%%%%%%%%%%%%%%%%%%%%%%%%%%%%%%%%%%%%%%%%%%%%%%%%%%%%%

We give next, sufficient conditions on $\cP$ for $\cM (\cP)$ to have no zero columns.
Let $S_{n-1}$ be the permutation group of $\{1,\ldots ,n-1\}$. A linear differential system $\cP$ is called {\sf differentially essential} if, there exist $i\in\{1,\ldots ,n\}$ and $\tau_i\in S_{n-1}$ such that
\begin{equation}\label{eq-taui}
\left\{
\begin{array}{ll}
\cL_{j,\tau_i(n-j)}\neq 0, & j=1,\ldots ,i-1,\\
\cL_{j,\tau_i(n-j+1)}\neq 0, &j=i+1,\ldots ,n.
\end{array}
\right.
\end{equation}
Observe that, if $\cP$ is differentially essential then $\nu(\cP)=n-1$ but the converse is false.
Differentially essential systems of generic, non necessarily linear, differential polynomials were defined in \cite{LGY}, Definition 3.3 and  \eqref{eq-taui} is a new characterization of this requirement in the case of linear differential polynomials.

\begin{defi}\label{def-complete}
A linear differential system $\cP$ is called {\sf super essential} if, for every $i\in\{1,\ldots ,n\}$, there exists $\tau_i\in S_{n-1}$ verifying \eqref{eq-taui}.
\end{defi}

The notion in Definition \ref{def-complete} is introduced for the first time
and its implications will be studied further in Section \ref{sec-irredundant systems of dpls}. Simultaneously,
in \cite{LYG} the notion of rank essential (non necessarily linear) system was introduced, and it is equivalent to super essential in the linear case.

Given a super essential system $\cP$, it will be
proved that $\dfres(\cP)$ can be defined and that the matrix $\cM(\cP)$ has no zero columns.
For this purpose, given $i\in\{1,\ldots ,n\}$, for every $\tau\in S_{n-1}$ we define bijections $\mu_{\tau}^i:\{1,\ldots ,n\}\backslash\{i\}\longrightarrow \{1,\ldots ,n-1\}$ by
\begin{equation}\label{eq-mutau}
\mu_{\tau}^i(j):=\left\{\begin{array}{ll}\tau(n-j),&j=1,\ldots ,i-1,\\\tau(n-j+1),&j=i+1,\ldots ,n.\end{array}\right.
\end{equation}
In particular, for $\tau_i$, $i=1,\ldots ,n$ as in Definition \ref{def-complete},
let
\begin{equation}\label{eq-mui}
\mu_i:=\mu_{\tau_i}^i, i=1,\ldots ,n.
\end{equation}

\begin{lem}\label{lem-exdFres}
Given a super essential system $\cP$, $N-o_i-\gamma(\cP)\geq 0$, $i=1,\ldots ,n$.
\end{lem}
\begin{proof}
Given $i\in\{1,\ldots ,n\}$,
\[N-o_i-\gamma(\cP)=\sum_{j\in\{1,\ldots ,n\}\backslash\{i\}}(o_j-\gamma_{\mu_i(j)}(\cP)).\]
By Definition \ref{def-complete} and \eqref{eq-mui}, $\cL_{j,\mu_i(j)}\neq 0$, $j\in\{1,\ldots ,n\}\backslash\{i\}$ and,
by \eqref{eq-oigammaj}, $o_j-\gamma_{\mu_i(j)}(\cP)\geq 0$. This proves the result.
\end{proof}

By Lemma \ref{lem-exdFres}, if $\cP$ is super essential then the differential resultant formula $\dfres(\cP)$ can be defined. Furthermore, we prove next that $\cM(\cP)$ has no zero columns although we cannot guarantee that $\dfres(\cP)\neq 0$ as Examples \ref{ex-dfres}(3) and \ref{ex-fin} show.

Given a linear differential polynomial $f\in\bbD\{U\}$,
\[f=a+\sum_{j=1}^{n-1}\cD_j(u_j),\mbox{ with }a\in\bbD\mbox{ and }\cD_j\in\bbD[\partial],\]
we define
$\frak{S}_j (f):=\frak{S}(\cD_j)$.
\begin{rem}\label{rem-k+1}
Let $f\in\bbD\{U\}$ be linear. If $k\in\frak{S}_j(f)$ but $k+1\notin\frak{S}_j(f)$ then $k+1\in \frak{S}_j(\partial f)$.
\end{rem}

\begin{thm}\label{thm-Mn0}
Given a super essential system $\cP$ (as in Section \ref{sec-Preliminary}), the matrix $\cM(\cP)$ has no zero columns.
\end{thm}
\begin{proof}
Equivalently, we will prove that $\ps(\cP)$ is a system of $L$ polynomials in the $L-1$ (algebraic) indeterminates of the set $\cV$ in \eqref{eq-V}.
We will prove that for every $j\in\{1,\ldots ,n-1\}$
\[\cup_{f\in\ps(\cP)} \frak{S}_j (f)=[\underline{\gamma}_j(\cP), N-\overline{\gamma}_j(\cP)-\gamma(\cP)]\cap\bbZ.\]
Let us denote $\gamma_j(\cP)$, $\overline{\gamma}_j(\cP)$ and $\underline{\gamma}_j(\cP)$, $j=1,\ldots ,n-1$ simply by $\gamma_j$, $\overline{\gamma}_j$ and $\underline{\gamma}_j$ respectively, in this proof.

Given $j\in\{1,\ldots ,n-1\}$, the set $\{\deg(\cL_{i,j})\mid \cL_{i,j}\neq 0, i=1,\ldots,n\}$ is not empty because $\nu(\cP)=n-1$. Let $I(j)\in \{1,\ldots ,n\}$ be such that
\begin{equation}\label{eq-D}
\deg(\cL_{I(j),j})=\min \{\deg(\cL_{i,j})\mid \cL_{i,j}\neq 0, i=1,\ldots,n\}
\end{equation}
and define $d:=\deg(\cL_{I(j),j})$.
Let $\overline{I}(j)\in \{1,\ldots ,n\}$ be such that
\[\overline{\gamma}_j=o_{\overline{I}(j)}-\deg(\cL_{\overline{I}(j),j}).\]
Observe that it may happen that $I(j)=\overline{I}(j)$ but not necessarily.
We can write
\[[\underline{\gamma}_j, N-\overline{\gamma}_j-\gamma]=[\underline{\gamma}_j,d-1]\cup [d,N-o_{I(j)}-\gamma+d ] \cup [N-o_{I(j)}-\gamma+d+1,N-\overline{\gamma}_j-\gamma].\]
If $\underline{\gamma}_j=d$ then the first interval is empty. If $I(j)=\overline{I}(j)$ then $d=o_{I(j)}-\overline{\gamma}_j$ and the third interval is empty.
\begin{enumerate}
\item For every $k\in [d,N-o_{I(j)}-\gamma+d ]\cap \bbZ$, since $d=\deg(\cL_{I(j),j})$, by Remark \ref{rem-k+1} we have $k\in\frak{S}_j (\partial^{k-d}f_{I(j)})$.

\item If $I(j)\neq \overline{I}(j)$. Given $k\in [N-o_{I(j)}-\gamma+d+1,N-\overline{\gamma}_j-\gamma]$. Let $\overline{d}:=\deg(\cL_{\overline{I}(j),j})=o_{\overline{I}(j)}-\overline{\gamma}_j$ and observe that
\[\overline{d} =o_{\overline{I}(j)}-\gamma_j+\underline{\gamma}_j\leq N-o_{I(j)}-\gamma+\underline{\gamma}_j\leq N-o_{I(j)}-\gamma+d <k \]
since $\underline{\gamma}_j\leq d$. The previous shows that
\[k-\overline{d}\leq N-\overline{\gamma}_j-\gamma -\overline{d}=N-o_{\overline{I}(j)}-\gamma.\]
By Remark \ref{rem-k+1}, $k\in\frak{S}_j(\partial^{k-\overline{d}} f_{\overline{I}(j)})$.

\item If $d\neq \underline{\gamma}_j$. Observe that $[\underline{\gamma}_j,d-1]\cap (\cup_{i=1}^n\frak{S}_j(f_i))\neq\emptyset$ because it contains $\underline{\gamma}_j$. For $k\in [\underline{\gamma}_j,d-1]\cap (\cup_{i=1}^n\frak{S}_j(f_i))$, there exists $i_k\in\{1,\ldots ,n\}$ such that $k\in\frak{S}_j(f_{i_k})$. If $[\underline{\gamma}_j,d-1]\backslash (\cup_{i=1}^n\frak{S}_j(f_i))\neq\emptyset$, given $k\in [\underline{\gamma}_j,d-1]\backslash (\cup_{i=1}^n\frak{S}_j(f_i))$, let
\[k'=\max [\underline{\gamma}_j,k-1]\cap (\cup_{i=1}^n\frak{S}_j(f_i)).\]
Observe that $k-k'\leq d-\underline{\gamma}_j$ and there exists $i_{k'}\in\{1,\ldots ,n\}$ such that $k'\in \frak{S}_j (f_{i_{k'}})$ but $k'+1\notin \frak{S}_j (f_{i_{k'}})$.
If $i_{k'}\neq I(j)$ then
\begin{equation}\label{eq-kkp1}
k-k'\leq o_{I(j)}-\gamma_j\leq N-o_{i_{k'}}-\gamma.
\end{equation}
Otherwise, $i_{k'}=I(j)$ and $k'\in \frak{S}_j(f_{I(j)})$. Since $\cP$ is super essential, the bijection $\mu_{I(j)}$ given by \eqref{eq-mui} is defined and for $I'=\mu_{I(j)}^{-1}(j)\in\{1,\ldots ,n\}\backslash \{I(j)\}$, $\cL_{I',j}\neq 0$. In particular, $\deg(\cL_{I',j})\leq o_{I'}-\overline{\gamma}_j$ and
\begin{equation}\label{eq-kkp2}
k-k'\leq d-\underline{\gamma}_j\leq o_{I'}-\gamma_j\leq N-o_{I(j)}-\gamma.
\end{equation}
By \eqref{eq-kkp1}, \eqref{eq-kkp2} and Remark \ref{rem-k+1} it is proved that $k\in \frak{S}_j (\partial^{k-k'} f_{i_{k'}})$.
\end{enumerate}
\end{proof}

\begin{ex}
\begin{enumerate}
\item Let us have a new look at the system of Example \ref{ex-notsupess},
\begin{equation}
\begin{array}{cccc}
f_1=a_1+& \cL_{1,1}(u_1) & + & 0\,\, ,\\
f_2=a_2+& 0 & + & \cL_{2,2}(u_2),\\
f_3=a_3+& 0 & + & \cL_{3,2}(u_2).
\end{array}
\end{equation}
It is differentially essential, namely for $i=3$ the permutation $\tau_3=(2\,\,1)\in S_2$ verifies \eqref{eq-taui}, $\cL_{1,\tau_3(2)}=\cL_{1,1}\neq 0$ and $\cL_{2,\tau_3(1)}=\cL_{2,2}\neq 0$.
This system is not super essential, for $i=1$ we cannot find $\tau_1\in S_2$ verifying \eqref{eq-taui}. The subsystem $\cP'=\{f_2,f_3\}$ in $\bbD\{u_2\}$ is super essential and $\dfres(\cP')$ is the determinant of a $4\times 4$ matrix while $\dfres(\cP)$ is the determinant of a matrix $14\times 14$.

\item The system in Example \ref{ex-motivation}(1) is super essential, let us construct $\cM(\cP)$.
Using the information in Example \ref{ex-gammas}, $L=\sum_{i=1}^3 (N-o_i-\gamma (\cP)+1)=11$ and
\begin{align*}
&\ps(\cP)=\{\partial^3 f_1,\partial^2 f_1 ,\partial f_1,f_1, \partial^2 f_2,\partial f_2,f_2,\partial^3 f_3,\partial^2 f_3 ,\partial f_3,f_3 \},\\
&\cV=\{u_{2,5},u_{2,4},u_{1,4},u_{2,3},u_{1,3},u_{2,2},u_{1,2},u_{2,1},u_{1,1},u_1\},
\end{align*}
whose elements are arranged in the order indexing the $L$ rows and first $L-1$ columns of $\cM(\cP)$ respectively. We show next the columns of the matrix $\cM(\cP)$, we denote $\partial^l a_{i,j,k}$ by $a_{i,j,k}^{(l)}$ and $\partial^l a_i$ by $a_i^{(l)}$, $l\in\bbN$, due to space limitations.
Observe that the first $L-1$ columns are the columns of the matrix $\cL(\ps(\cP),\cV)$ in \eqref{eq-Lps}.
\[
\begin{array}{rccccc}
 &u_{2,5}&u_{2,4}&u_{1,4}&u_{2,3}&u_{1,3}\\
 & & & & & \\
\partial^3 f_1\rightarrow &a_{1,2,2} &  a_{1,2,1}+3a_{1,2,2}^{(1)} & a_{1,1,1}& 3a_{1,2,1}^{(1)}+3a_{1,2,2}^{(2)} & a_{1,1,0}+3a_{1,1,1}^{(1)}
\\
\partial^2 f_1 \rightarrow & 0 &  a_{1,2,2} & 0 & a_{1,2,1}+2a_{1,2,2}^{(1)}& a_{1,1,1}
\\
\partial f_1\rightarrow & 0& 0& 0& a_{1,2,2}& 0
\\
f_1\rightarrow & 0& 0& 0& 0& 0
\\
\partial^2 f_2 \rightarrow & a_{2,2,3}& a_{2,2,2}+2a_{2,2,3}^{(1)}& 0& 2a_{2,2,2}^{(1)}+a_{2,2,3}^{(2)}& 0
\\
\partial^1 f_2 \rightarrow & 0 &  a_{2,2,3} & 0& a_{2,2,2}+a_{2,2,3}^{(1)}& 0
\\
f_2 \rightarrow & 0&0& 0& a_{2,2,3}& 0
\\
\partial^3 f_3 \rightarrow & a_{3,2,2}&a_{3,2,1}+3a_{3,2,2}^{(1)}& a_{3,1,1}^{(1)}& 3a_{3,2,1}+3a_{3,2,2}^{(2)}& 3a_{3,1,1}^{(1)}
\\
\partial^2 f_3 \rightarrow & 0& a_{3,2,2} & 0& a_{3,2,1}+2a_{3,2,2}^{(1)}& a_{3,1,1}
\\
\partial f_3 \rightarrow & 0& 0& 0& a_{3,2,2}& 0
\\
f_3 \rightarrow & 0& 0& 0& 0& 0
\end{array}
\]
\[
\begin{array}{cccccc}
u_{2,2}& u_{1,2}&u_{2,1}&u_{1,1}&u_1&1\\
 & & & & \\
3a_{1,2,1}^{(2)}+a_{1,2,2}^{(3)}  & 3a_{1,1,0}^{(1)}+3a_{1,1,1}^{(2)} & a_{1,2,1}^{(3)} &  3a_{1,1,0}^{(2)}+a_{1,1,1}^{(3)} &
a_{1,1,0}^{(3)} &a_1^{(3)}
\\
2a_{1,2,1}^{(1)}+a_{1,2,2}^{(2)} & a_{1,1,0}+2a_{1,1,1}^{(1)}& a_{1,2,1}^{(2)} & 2a_{1,1,0}^{(1)}+a_{1,1,1}^{(2)}& a_{1,1,0}^{(2)} & a_1^{(2)}
\\
a_{1,2,1}+a_{1,2,2}^{(1)} & a_{1,1,1}& a_{1,2,1}^{(1)}& a_{1,1,0}+a_{1,1,1}^{(1)}& a_{1,1,0}^{(1)} & a_1^{(1)}
\\
a_{1,2,2} & 0& a_{1,2,1}& a_{1,1,1}& a_{1,1,0} & a_1
\\
a_{2,2,2}^{(2)} & 0& 0& 0&0 &a_2^{(2)}
\\
a_{2,2,2}^{(1)} & 0& 0& 0& 0  &a_2^{(1)}
\\
a_{2,2,2} & 0& 0& 0& 0  &a_2
\\
3a_{3,2,1}^{(2)}+a_{3,2,2}^{(3)} & 3a_{3,1,1}^{(2)}& a_{3,2,1}^{(3)}& a_{3,1,1}^{(3)}& 0 &a_3^{(3)}
\\
2a_{3,2,1}^{(1)}+a_{3,2,2}^{(2)} & 2a_{3,1,1}^{(1)}& a_{3,2,1}^{(2)}& a_{3,1,1}^{(2)} & 0  & a_3^{(2)}
\\
a_{3,2,1}+a_{3,2,2}^{(1)} & a_{3,1,1}& a_{3,2,1}^{(1)}& a_{3,1,1}^{(1)}& 0  & a_3^{(1)}
\\
a_{3,2,2} & 0& a_{3,2,1}& a_{3,1,1}& 0 & a_3
\end{array}
\]
\end{enumerate}
\end{ex}

%%%%%%%%%%%%%%%%%%%%%%%%%%%%%%%%%%%%%%%%%%%%%%%%%%%%%%%%%%%%%%%%%%%
\section{Irredundant systems of linear differential polynomials}\label{sec-irredundant systems of dpls}
%%%%%%%%%%%%%%%%%%%%%%%%%%%%%%%%%%%%%%%%%%%%%%%%%%%%%%%%%%%%%%%%%%%

A key fact to eliminate the differential variables in $U$ from the system $\cP$ is that not all the polynomials in $\cP$ have to be involved in the computation. Namely only the polynomials in a super essential subsystem $\cP^*$ of $\cP$ are needed to achieve the elimination. In this section, it is proved that every system $\cP$ contains a super essential subsystem $\cP^*$ and
a new characterization of differentially essential systems is given, namely they are the system having a unique super essential subsystem.

The linear differential system $\cP$ is an overdetermined system, in the differential variables $U$.
Recall that we assumed $\nu(\cP)=n-1=|\cP|-1$.
It is proved in this section that, the super essential condition on $\cP$ is equivalent with
every proper subsystem $\cP'$ of $\cP$ not being overdetermined, in the differential variables $U$.
A name for this idea seems to be lacking in the literature.
\begin{defi}
A system of linear differential polynomials $\cP$ is called {\sf irredundant} (for differential elimination purposes),
if every proper subsystem $\cP'$ of $\cP$ verifies
$|\cP'|\leq \nu(\cP')$. Otherwise, $\cP$ is called {\sf redundant}.
\end{defi}

Furthermore, it will be shown in this section that every linear differential system $\cP$ (even if it is not differentially essential) contains a super essential subsystem $\cP^*$. Let $\cP_i:=\cP\backslash \{f_i\}$.

\begin{prop}\label{prop-irse}
If $\cP$ is super essential then $\cP$ is irredundant.
\end{prop}
\begin{proof}
For every proper subset $\cP'=\{f_{h_1},\ldots ,f_{h_m}\}$ of $\cP$, there exists $i\in\{1,\ldots ,n\}$ such that $\cP'\subseteq\cP_i$. Therefore $h_1,\ldots,h_m\in\{1,\ldots ,n\}\backslash\{i\}$ and given $\mu_i$ as in \eqref{eq-mui},
\[\cL_{h_t,\mu_i(h_t)}\neq 0,\, t=1,\ldots ,m.\]
Since $\mu_i$ is a bijection, $\nu(\cP')\geq m=|\cP'|$.
\end{proof}

Let $x_{i,j}$, $i=1,\ldots ,n$, $j=1,\ldots ,n-1$ be algebraic indeterminates over $\bbQ$, the field of rational numbers.
Let $X(\cP)=(X_{i,j})$ be the $n\times (n-1)$ matrix, such that
\begin{equation}\label{eq-XP}
X_{i,j}:=\left\{\begin{array}{ll}x_{i,j},&\cL_{i,j}\neq 0,\\ 0,&\cL_{i,j}=0.\end{array}\right.
\end{equation}
We denote by $X_i(\cP)$, $i=1,\ldots ,n$, the submatrix of $X(\cP)$ obtained by removing its $i$th row.
Thus $X(\cP)$ is an $n\times (n-1)$ matrix with entries in the field $\bbK:=\bbQ(X_{i,j}\mid X_{i,j}\neq 0)$.

\begin{lem}\label{lem-Xi}
Given $i\in\{1,\ldots ,n\}$, $\det(X_i(\cP))\neq 0$ if and only if there exists $\tau_i\in S_{n-1}$ verifying \eqref{eq-taui}.
\end{lem}
\begin{proof}
Given $\tau\in S_{n-1}$, let us consider the bijection $\mu_{\tau}:=\mu_{\tau}^i$ as in \eqref{eq-mutau}.
We can write
\begin{equation}\label{eq-detXi}
\det(X_i(\cP))=\sum_{\tau\in S_{n-1}} \prod_{j\in\{1,\ldots ,n\}\backslash\{i\}} X_{j,\mu_{\tau}(j)}.
\end{equation}
The entries of $X_i(\cP)$ are either algebraic indeterminates or zero. Thus
$\det(X_i(\cP))=0$ if and only if every summand of \eqref{eq-detXi} is zero, it contains a zero entry. That is, for every $\tau\in S_{n-1}$, there exists $j\in \{1,\ldots ,n\}\backslash\{i\}$ such that $X_{j,\mu_{\tau}(j)}=0$, thus $\cL_{j,\mu_{\tau}(j)}=0$.  This proves that, $\det(X_i(\cP))=0$ if and only if there is no $\tau\in S_{n-1}$ verifying \eqref{eq-taui}.
\end{proof}

\begin{rem}\label{rem-de-se}
From Lemma \ref{lem-Xi} we can conclude that:
\begin{enumerate}
\item $\cP$ is differentially essential $\Leftrightarrow$ $\rank(X(\cP))=n-1$.

\item $\cP$ is super essential $\Leftrightarrow \det(X_i(\cP))\neq 0$, $i=1,\ldots ,n$.
\end{enumerate}
\end{rem}

Given the set $\bbP:=\{p_1,\ldots ,p_n\}$ of algebraic polynomials in $\bbK[C,U]$, $\bbK=\bbQ(X_{i,j}\mid X_{i,j}\neq 0)$, with
\[p_i:=c_i+\sum_{j=1}^{n-1}X_{i,j} u_j,\, i=1,\ldots, n,\]
a coefficient matrix $M(\bbP)$ of $\bbP$ is an $n\times (2n-1)$ matrix and it can be obtained by concatenating $X(\cP)$ with the identity matrix of size $n$,
\begin{equation}\label{eq-MbbP}
M(\bbP)=\left[\begin{array}{cccc}
\,&1&\cdots &0\\
X(\cP)&\,&\ddots&\,\\
\,&0&\cdots &1
\end{array}\right].
\end{equation}
The reduced echelon form of $M(\bbP)$ is the coefficient matrix of the reduced Gr\"{o}bner basis $\cB=\{e_0,e_1,\ldots ,e_{n-1}\}$ of the algebraic ideal $(\bbP)$ generated by $p_1,\ldots,p_n$ in $\bbK[C,U]$, with respect to lex monomial order with $u_1>\cdots >u_{n-1}>c_1>\cdots>c_n$ (\cite{Cox}, p. 95, Exercise 10). We assume that $e_0<e_1<\cdots<e_{n-1}$.

Observe that the elements of $\cB$ are linear homogeneous polynomials in $\bbK[C,U]$ and at least
\begin{equation}\label{eq-B0}
e_0\in\cB_0:=\cB\cap\bbK[C].
\end{equation}
Given a linear homogeneous polynomial $e\in\bbK[C]$, $e=\sum_{h=1}^n \chi_h c_h$, $\chi_h\in\bbK$, let $\cI(e):=\{h\in\{1,\ldots ,n\}\mid \chi_h\neq 0\}$.
Let us consider the system
\begin{equation}\label{eq-P*}
\cP^*:=\{f_h\mid h\in \cI(e_0)\}.
\end{equation}

\begin{rem}\label{rem-Idese}
Let $\cI:=\{i\in\{1,\ldots ,n\}\mid \det(X_i(\cP))=0\}$. By Remark \ref{rem-de-se} the following statements hold
\begin{equation}\label{eq-de-se}
\begin{array}{l}
\cP\mbox{ is differentially essential }\Leftrightarrow \cI\neq \{1,\ldots ,n\}\mbox{ and}\\
\cP\mbox{ is super essential }\Leftrightarrow \cI=\emptyset.
\end{array}
\end{equation}
Furthermore, if $\cP$ is differentially essential then, by Remark \ref{rem-de-se}, $\cB_0=\{e_0\}$ and by \eqref{eq-de-se}, up to a nonzero constant,
\begin{equation}\label{eq-e0}
e_0=\sum_{i\in \cI(e_0)} \det(X_i(\cP)) c_i,\mbox{ with }\cI(e_0)=\{1,\ldots ,n\}\backslash \cI,
\end{equation}
the determinant of the matrix obtained by concatenating $X(\cP)$ with the column vector containing $c_1,\ldots ,c_n$.
\end{rem}

\begin{lem}\label{Lem-P*proper}
If $\cP$ is super essential then $\cP=\cP^*$, otherwise $\cP^*\varsubsetneq\cP$.
\end{lem}
\begin{proof}
By Remark \ref{rem-Idese}, if $\cP$ is super essential $\cI(e_0)=\{1,\ldots ,n\}$, that is $\cP^*=\cP$. Otherwise, $\cI\neq \emptyset$ and we have two possibilities: if $\cI\neq \{1,\ldots ,n\}$ then, by \eqref{eq-e0}, $\cI(e_0)\varsubsetneq\{1,\ldots ,n\}$; if $\cI=\{1,\ldots ,n\}$ then, by Remark \ref{rem-de-se}(1) and \eqref{eq-de-se}, $\rank(X(\cP))<n-1$ and $e_1\in \cB_0$ with $e_0<e_1$, which implies $\cI(e_0)\varsubsetneq\{1,\ldots ,n\}$.
\end{proof}

We will prove next that $\cP^*$ is a super essential subsystem of $\cP$.

\begin{lem}\label{lem-P*}
\begin{enumerate}
\item For every $\cP'\varsubsetneq\cP^*$, $\rank(X(\cP'))=|\cP'|$.

\item $\rank(X(\cP^*))=|\cP^*|-1$.
\end{enumerate}
\end{lem}
\begin{proof}
\begin{enumerate}
\item Given a proper subsystem $\cP'$ of $\cP$, the matrix $X(\cP')$ has size $|\cP'|\times (n-1)$. Thus $\rank(X(\cP'))\leq |\cP'|$.
The coefficient matrix in $\bbK[C',U]$ of $\bbP':=\{p_h\mid f_h\in\cP'\}$, with $C'=\{c_h\mid f_h\in\cP'\}$, is
\[M(\bbP')=\left[\begin{array}{cccc}
\,&1&\cdots &0\\
X(\cP')&\,&\ddots&\,\\
\,&0&\cdots &1
\end{array}\right].\]
%If $\cP'\varsubsetneq\cP^*$ then $\bbP' \varsubsetneq \bbP^*:=\{p_h\mid h\in I(e_0)\}$.

If $\rank(X(\cP'))< |\cP'|$ then there exists $e\in (\bbP')\cap\bbK[C']$, the vector whose coefficients are in the last row of the reduced echelon form of $M(\bbP')$. Therefore, $\cP'\varsubsetneq\cP^*$ together with \eqref{eq-P*} imply
\begin{equation}\label{eq-ie0}
\cI(e)\subseteq \{h\in\{1,\ldots ,n\}\mid f_h\in\cP'\}\varsubsetneq \cI(e_0),
\end{equation}
and $e\in (\bbP')\cap\bbK [C]\subset (\bbP)\cap\bbK [C]=(\cB_0)$.
This contradicts that $\cB_0$ is a Gr\"{o}bner basis of $(\bbP)\cap\bbK[C]$, since $\remm(e,\cB_0)$, the remainder of the division of $e$ by $\cB_0$, equals by \eqref{eq-ie0} $\remm(e,e_0)\neq 0$. Therefore $\rank(X(\cP'))= |\cP'|$.

\item Let $m=|\cP^*|$ and $\bbP^*:=\{p_h\mid h\in \cI(e_0)\}$. By 1, $\rank(X(\cP'))=m-1$ for every $\cP'\varsubsetneq \cP^*$ with $|\cP'|=m-1$. Thus $\rank(X(\cP^*))\geq m-1$ because $X(\cP')$ is a submatrix of $X(\cP^*)$.
On the other hand $e_0\in (\bbP^*)\cap\bbK[C]$ implies $\rank(X(\cP^*))<m$, otherwise the reduced echelon form of $M(\bbP^*)$ provides no vector in $\bbK [C]$. Therefore $\rank(X(\cP^*))=m-1$.
\end{enumerate}
\end{proof}

Given a proper subsystem $\cP'=\{g_1:=f_{i_1},\ldots ,g_m:=f_{i_m}\}$ of $\cP$ and $J=\{j_1,\ldots ,j_{m-1}\}\subset \{1,\ldots ,n-1\}$,
let $Y^J(\cP')$ be the $m\times (m-1)$ matrix
\begin{equation}\label{eq-YP}
Y^J(\cP'):=(Y_{h,k}),\,\,\, Y_{h,k}:=X_{i_h,j_k}, h=1,\ldots ,m, k=1,\ldots ,m-1.
\end{equation}
Denote by $Y^J_h(\cP')$ the submatrix of $Y^J(\cP')$ obtained by removing the $h$th row, $h=1,\ldots ,m$.
If $\cP'$ is super essential then there exists $J\subset \{1,\ldots ,n-1\}$, $|J|=m-1$ such that:
\begin{equation}\label{eq-gh}
g_h=a_{i_h}+\sum_{j\in J} \cL_{i_h,j}(u_j), h=1,\ldots ,m,
\end{equation}
and
\begin{equation}\label{eq-detYh}
\det(Y^J_h(\cP'))\neq 0,\, h=1,\ldots ,m.
\end{equation}
That is, $\nu(\cP')=|\cP'|-1$ and Remark \ref{rem-de-se}(2) is verified.

\begin{thm}\label{thm-subSC}
If $\cP$ is not super essential then, the system $\cP^*$ given by \eqref{eq-P*} is a proper super essential subsystem of $\cP$, with $\nu(\cP^*)=|\cP^*|-1$.
\end{thm}
\begin{proof}
We can write $\cP^*=\{g_1:=f_{i_1},\ldots ,g_m:=f_{i_m}\}$.
By Lemma \ref{lem-P*}, there exists $J=\{j_1,\ldots ,j_{m-1}\}\varsubsetneq\{1,\ldots ,n-1\}$, such that
\begin{equation}\label{eq-detYm}
\det(Y^J_m(\cP^*))\neq 0.
\end{equation}
Let us denote $Y^J(\cP^*)$ simply by $Y(\cP^*)$ and $Y^J_h(\cP^*)$ by $Y_h(\cP^*)$, $h=1,\ldots ,m$, in the remaining parts of the proof.
Observe that $Y(\cP^*)$ is a submatrix of $X(\cP^*)$.
We will prove that, the only nonzero entries of $X(\cP^*)$ are the ones in the submatrix $Y(\cP^*)$, that is \eqref{eq-gh} is verified or equivalently
\begin{equation}\label{eq-proof1}
p_{i_h}=c_{i_h}+\sum_{k=1}^{m-1} Y_{h,k} u_{j_k}, h=1,\ldots ,m,
\end{equation}
and also
\begin{equation}\label{eq-proof2}
\det(Y_h(\cP^*))\neq 0,\, h=1,\ldots ,m.
\end{equation}

For this purpose, we will prove the following claims. For $l\in\{1,\ldots ,m\}$, if $\det(Y_l(\cP^*))\neq 0$ then
\begin{equation}\label{eq-gl}
p_{i_l}=c_{i_l}+\sum_{k=1}^{m-1}Y_{l,k} u_{j_k}\mbox{ and }
\end{equation}
there exists a bijection $\eta_{l}:\{1,\ldots ,m\}\backslash\{l\}\longrightarrow \{1,\ldots ,m-1\}$ such that
\begin{equation}\label{eq-Yt}
\det(Y_t(\cP^*))\neq 0,\, \forall t\in T_l:=\{t\in\{1,\ldots ,m\}\backslash\{l\}\mid Y_{l,\eta_l(t)}\neq 0\}.
\end{equation}
\begin{enumerate}
\item Proof of \eqref{eq-gl}. Otherwise, there exists $j\in\{1,\ldots ,n-1\}\backslash J$ such that $X_{i_l ,j}\neq 0$.
This means that the matrix
\[\left[\begin{array}{cc}
\,& X_{i_1,j}\\
Y(\cP^*)&\vdots\\
\,&X_{i_m,j}\end{array}\right],\]
is nonsingular, which contradicts $\rank(X(\cP^*))=m-1$, see Lemma \ref{lem-P*}.

\item Proof of \eqref{eq-Yt}. Since $\det(Y_l(\cP^*))\neq 0$, by Lemma \ref{lem-Xi}, there exists $\tau_l\in S_{m-1}$ and a
 bijection
\begin{align*}
&\eta_{l}:\{1,\ldots ,m\}\backslash\{l\}\longrightarrow \{1,\ldots ,m-1\},\\
&\eta_{l}(h):=\left\{\begin{array}{ll}\tau_l(m-h),&h=1,\ldots ,l-1,\\\tau_l(m-h+1),&h=l+1,\ldots ,m,\end{array}\right.
\end{align*}
such that
\begin{equation}\label{eq-Ytaum}
Y_{h,\eta_l(h)}\neq 0,\, h\in \{1,\ldots ,m\}\backslash\{l\}.
\end{equation}
Given $t\in T_l$ and the permutation $\rho(l,t):\{1,\ldots ,m\}\longrightarrow \{1,\ldots ,m\}$, such that
\[\rho(l,t)(h)=\left\{
\begin{array}{ll}
t,& h=l,\\
l,& h=t,\\
h,& h\in \{1,\ldots ,m\}\backslash\{t,l\},
\end{array}
\right.
\]
we define the bijection
\begin{equation}\label{eq-mut}
\eta_t:\{1,\ldots ,m\}\backslash\{t\}\longrightarrow \{1,\ldots ,m-1\}, \eta_{t}=\eta_l\circ \rho(l,t).
\end{equation}
Thus, by \eqref{eq-Ytaum} and the definition of $T_l$, $Y_{h,\eta_t(h)}\neq 0$, $h\in\{1,\ldots ,m\}\backslash\{t\}$, which proves that $\det(Y_t(\cP^*))\neq 0$.
\end{enumerate}

We are ready to prove \eqref{eq-proof1} and \eqref{eq-proof2}.
By \eqref{eq-detYm} and \eqref{eq-Yt}, already \eqref{eq-proof2} holds for $h\in T_m\cup \{m\}$ and, by \eqref{eq-gl}, \eqref{eq-proof1} holds for $h\in T_m\cup \{m\}$. We follow the next loop to prove \eqref{eq-proof1} and \eqref{eq-proof2} for $h\in \{1,\ldots ,m-1\}\backslash T_m$.

\begin{enumerate}
\item Set $T:=T_m$ and $\cP':=\{g_h\mid h\in T\cup\{m\}\}$.

\item If $T=\{1,\ldots ,m-1\}$ then $\cP^*=\cP'$, which proves \eqref{eq-proof1} and \eqref{eq-proof2}, by \eqref{eq-detYm}, \eqref{eq-Yt} and \eqref{eq-gl}.

\item If $T\neq \{1,\ldots ,m-1\}$ then, there exists $l\in T$ such that $T_l\backslash (T\cup \{m\})\neq\emptyset$ (see below). Set $T:=(T\cup T_l)\backslash \{m\}$, $\cP':=\{g_h\mid h\in T\cup\{m\}\}$ and observe that by \eqref{eq-Yt}, \eqref{eq-proof2} holds for $h\in T\cup \{m\}$ and by \eqref{eq-gl}, \eqref{eq-proof1} holds for $h\in T\cup \{m\}$. Go to step 2.
\end{enumerate}
We prove next that the loop finishes because each time we go to step 3 at least one new element is added to $T$. More precisely, we prove that (in the situation of step 3) there exists $l\in T$ such that $T_l\backslash (T\cup \{m\})\neq\emptyset$.
%in fact after the $p$th iteration
%\[T=(T_m\cup T_{l_1}\cup\ldots \cup T_{l_p})\backslash{\{m\}},\]
%with $l_i\in (T_m\cup T_{l_1}\cup\ldots \cup T_{l_{i-1}})\backslash{\{m\}}$, $i=2,\ldots ,p$ and $l_1\in T_m$.
We assume that at some iteration $T_l\subseteq T\cup\{m\}$, $\forall l\in T$ to reach a contradiction. Given $l\in T$, if $l\in T_m$, by \eqref{eq-mut}, $\eta_l=\eta_m\circ \rho(m,l)$ and
\[\{\eta_l(t)\mid t\in T_l\}\subseteq \{\eta_m(t)\mid t\in T\},\]
else there exist $l_1,\ldots ,l_p\in\{1,\ldots ,m-1\}$ such that
$l\in T_{l_p}$, $l_k\in T_{l_{k-1}}$, $k=2,\ldots ,p$ and $l_1\in T_m$, by \eqref{eq-mut}
%\[\eta_l=\eta_{l_p}\circ\rho(l_p,l)=\eta_{l_{p-1}}\circ %\rho(l_{p-1},l_p)\circ\rho(l_p,l)=\eta_{l_{p-1}}\circ\rho(l_{p-1},l)=\cdots =\eta_{m}\circ\rho(m,l),\]
\begin{align*}
&\{\eta_l(t)\mid t\in T_l\}\subseteq \{\eta_{l_p}(t)\mid t\in T\backslash{\{l_p\}}\}=\\
&\{\eta_{l_{p-1}}(t)\mid t\in T\backslash{\{l_{p-1}\}}\}=\cdots =\{\eta_{l_1}(t)\mid t\in T\backslash{\{l_1\}}\}\subseteq\{\eta_{m}(t)\mid t\in T\}.
\end{align*}
%with $l_{i_j}\in (T_m\cup T_{l_1}\cup\ldots \cup T_{l_{i_j-1}})\backslash{\{m\}}$, $i_1\geq i_2\geq\cdots \geq i_q$ and $l_{i_q}\in T_m$.
By definition of $\cP'$ and $T_l$, we have proved that $\nu(\cP')\leq |T|$
and thus $\rank(X(\cP'))\leq \nu(\cP')\leq |T|$, contradicting Lemma \ref{lem-P*} since $\cP'\varsubsetneq \cP^*$ and $|\cP'|=|T|+1$.
\end{proof}

In particular, Theorem \ref{thm-subSC} shows that if $\cP$ is not super essential then $\cP$ is redundant,
 which together with Proposition \ref{prop-irse} proves the next result.

\begin{cor}
A linear differential system $\cP$ is irredundant if and only if it is super essential.
\end{cor}

The next result shows that if $\cP$ is differentially essential then $\cP^*$ is in fact the only super essential subsystem of $\cP$.
This new characterization of differentially essential systems (in the linear case) has now a flavor similar to the essential condition in the algebraic case, see \cite{St}, Section 1.

\begin{thm}\label{thm-seunique}
$\cP$ is differentially essential if and only if $\cP$ has a unique super essential subsystem.
\end{thm}
\begin{proof}
\begin{enumerate}
\item If $\cP$ is differentially essential, by Remark \ref{rem-de-se} and \eqref{eq-B0}, $(\bbP)\cap\bbK[C]=(e_0)$.
By Theorem \ref{thm-subSC}, $\cP^*$ is super essential.
Let us assume that there exists a super essential subsystem $\cP'=\{f_{t_1},\ldots ,f_{t_s}\}$ of $\cP$ different from $\cP^*$. This means that
$\{t_1,\ldots,t_s\}\neq \cI(e_0)=\{i_1,\ldots ,i_m\}$.
Let $\bbP':=\{p_i\mid f_i\in\cP'\}$, by \eqref{eq-detYh} and \eqref{eq-e0}, $(\bbP')\cap\bbK [C]=(e)$ with
\begin{align*}
e=&\sum_{l=1}^s \det(Y^K_{l}(\cP')) c_{t_l},\mbox{ and every }\det(Y^K_{l}(\cP'))\neq 0,
\end{align*}
for $K\subset \{1,\ldots ,n-1\}$, $|K|=s-1$.
This contradicts that $e\in (\bbP)\cap\bbK[C]=(e_0)$ because
$\cI(e)=\{t_1,\ldots,t_s\}\neq \cI(e_0)$.

\item Conversely, if $\cP$ is not differentially essential then, by Remark \ref{rem-de-se}, $\rank (X(\cP))<n-1$. This implies that the initial variable $c_{\iota}$, $\iota\in\{1,\ldots ,n\}$ of $e_0$ w.r.t. the order $c_1>c_2>\cdots >\c_n$ verifies $\iota\geq 2$ because $e_0$ is obtained from \eqref{eq-MbbP}. Let $\rho\in S_n$ be the permutation of $1$ and $\iota$. By the same reasoning, if we compute the reduced Gr\" obner basis $\cB'=\{e'_0,\ldots ,e'_{n-1}\}$ of $(\bbP)$ w.r.t. lex monomial order, with
    \[u_1>\cdots >u_{n-1}>c_{\iota}>c_{\rho(2)}>\cdots >c_{\rho(n)}\]
    and $e'_0<\cdots <e'_{n-1}$, then the initial variable of $e'_0$ is not $c_{\iota}$. Thus $\cI(e'_0)\neq \cI(e_0)$ and by Theorem \ref{thm-subSC}, $\{f_i\mid i\in \cI(e'_0)\}$ is also a super essential subsystem of $\cP$, different from $\cP^*$.
\end{enumerate}
\end{proof}

\begin{exs}\label{exs-se-de}
\begin{enumerate}
\item Given the system $\cP=\{f_1=\cL_{1,1}(u_1)+\cL_{1,2}(u_2),f_2=\cL_{2,1}(u_1),f_3=\cL_{3,2}(u_2)\}$ the matrix $X(\cP)$ defined by \eqref{eq-XP} equals
\[
X(\cP)=\left[\begin{array}{cc}
x_{1,1} & x_{1,2}\\
x_{2,1} & 0 \\
0 & x_{3,2}
\end{array}\right].
\]
By Remark \ref{rem-de-se}, $\cP$ is supper essential and, by Lemma \ref{Lem-P*proper}, $\cP^*=\cP$.

\item Let $\cP$ be a system such that
\[
X(\cP)=\left[\begin{array}{ccc}
x_{1,1} & x_{1,2} & 0\\
x_{2,1} & 0 & x_{2,3}\\
0 & x_{3,2} & 0\\
0 & x_{4,2} & 0
\end{array}\right].
\]
By Remark \ref{rem-de-se}, $\cP$ is differentially essential but it is not super essential. The reduced echelon form of the matrix $M(\bbP)$ in \eqref{eq-MbbP} is
\[
E=\left[\begin{array}{ccccccc}
x_{1,1} & x_{1,2} & 0 & 1 & 0 & 0 & 0\\
0 & x_{3,2} & 0 & 0 & 0 & 1 & 0\\
0 & 0 & x_{2,3} & -x_{2,1}/x_{1,1} & 1 & x_{2,1} x_{1,2}/ x_{1,1} x_{3,2} & 0\\
0 & 0 & 0 & 0 & 0 & -x_{4,2}/x_{3,2} & 1
\end{array}\right].
\]
The columns of $E$ are indexed by $u_1 > u_2 > u_3 > c_1 > c_2 > c_3 > c_4$ and its last row gives the coefficients of $e_0$, see \eqref{eq-B0}. Thus $\cI(e_0)=\{3,4\}$ and $\cP^*=\{f_3,f_4\}$ is the only super essential subsystem of $\cP$, by Theorems \ref{thm-subSC} and \ref{thm-seunique}.

\item Let $\cP$ be a system such that
\[
X(\cP)=\left[\begin{array}{ccc}
x_{1,1} & x_{1,2} & x_{1,3}\\
0 & x_{2,2} & 0\\
0 & x_{3,2} & 0\\
0 & x_{4,2} & 0
\end{array}\right].
\]
By Remark \ref{rem-de-se}, $\cP$ is not differentially essential and thus it is not super essential either. The reduced echelon form of the matrix $M(\bbP)$ in \eqref{eq-MbbP} is
\[
E=\left[\begin{array}{ccccccc}
x_{1,1} & x_{1,2} & x_{1,3} & 1 & 0 & 0 & 0\\
0 & x_{2,2} & 0 & 0 & 1 & 0 & 0\\
0 & 0 & 0 & 0 & -x_{3,2}/x_{2,2} & 1 & 0\\
0 & 0 & 0 & 0 & 0 & -x_{4,2}/x_{3,2} & 1
\end{array}\right].
\]
The columns of $E$ are indexed by $u_1 > u_2 > u_3 > c_1 > c_2 > c_3 > c_4$ and its last two rows give the coefficients of $e_0<e_1$ such that $\cB_0=\{e_0,e_1\}$, see \eqref{eq-B0}. Thus $\{f_3,f_4\}$ is a super essential subsystem of $\cP$ but in this case $\{f_2,f_3\}$ and $\{f_2,f_4\}$ are also super essential subsystems of $\cP$.
%Any of them can be described by the smallest element of a reduced Gr\"{o}bner basis of $(\bbP)\cap \bbK [C]$ w.r.t. the appropriate ordering of the variables %in $C=\{c_1,c_2,c_3,c_4\}$.
\end{enumerate}
\end{exs}

%%%%%%%%%%%%%%%%%%%%%%%%%%%%%%%%%%%%%%%%%%%%%%%%%%%%%%%%%%%%%%%%%%%
\section{Differential elimination for systems of linear DPPEs}\label{sec-linDPPEs}
%%%%%%%%%%%%%%%%%%%%%%%%%%%%%%%%%%%%%%%%%%%%%%%%%%%%%%%%%%%%%%%%%%%

In this section, we set $\bbD=\cK\{C\}$ and consider a system of linear differential polynomials in $\bbD\{U\}=\cK\{C,U\}$,
\begin{equation}\label{eq-PF}
\cP=\{F_i:=c_i-H_i(U), i=1,\ldots, n\},
\end{equation}
with $-H_i(U)=\sum_{j=1}^{n-1} \cL_{i,j}(u_j)$, $\cL_{i,j}\in\cK[\partial]$.
Observe that $\cP$ verifies ($\cP$2) and ($\cP$3) (as in Section \ref{sec-Preliminary}), let us assume ($\cP$1) and ($\cP$4).
Let $[\cP]_{\bbD\{U\}}$ be the differential ideal generated by $\cP$ in $\bbD \{U\}$.
By \cite{Gao}, Lemmas 3.1 and 3.2, $[\cP]_{\bbD\{U\}}$ is a differential prime ideal whose elimination ideal in $\bbD$ equals
\[\id (\cP):=[\cP]_{\bbD\{U\}}\cap\bbD=\{f\in\bbD\mid f(H_1(U),\ldots ,H_n (U))=0\}.\]
It is called in \cite{Gao} the implicit ideal of the system of linear differential polynomial parametric equations (linear DPPEs)
\begin{equation*}\label{DPPEs}
\left\{\begin{array}{ccc}c_1 &= & H_1 (U),\\ & \vdots  & \\
c_n&= & H_n (U).\end{array}\right.
\end{equation*}
Let $\PS\subset\partial\cP$ and $\cU\subset\{U\}$ be sets verifying $(\ps1)$ and $(\ps2)$ (as in Section \ref{sec-diff res formulas} but with $\cP$ as in \eqref{eq-PF}). The set $\PS$ belongs to the polynomial ring $\cK[\cC_{\PS},\cU]$, with
\[\cC_{\PS}:=\{c_{i,k}\mid k\in [0,L_i]\cap\bbZ,\, i=1,\ldots ,n\}.\]
Let $(\PS)$ be the algebraic ideal generated by $\PS$ in $\cK[\cC_{\PS},\cU]$.

The implicitization of linear DPPEs by differential resultant formulas was studied in \cite{RS} and \cite{R11}.
The results in \cite{RS} and \cite{R11} were written for specific choices of $\PS$ and $\cU$, as described in Remark \ref{rem-psuR}.
In this section, some of the results in \cite{R11} are presented for general $\PS$ and $\cU$, to be used in Section \ref{sec-Computation of the sparse}, namely Theorem \ref{thm-dimid}. In addition, the perturbation methods in \cite{R11}, Section 6 are extended to be used with formula $\dfres (\cP)$. We also emphasize on the relation between the implicit ideal of $\cP$ and the implicit ideals of its subsystems.

Let $\cP'$ be a subsystem of $\cP$.
If $|\cP'|=m$ then $\cP'=\{F_{h_1},\ldots ,F_{h_m}\}$ and the implicit ideal of $\cP'$ equals
\begin{equation}\label{eq-idPp}
\id (\cP')=\{f\in\cK\{C'\}\mid f(H_{h_1}(U),\ldots ,H_{h_m} (U))=0\},
\end{equation}
where $C'=\{c_i\mid F_i\in\cP'\}$. Let $\bbD':=\cK\{C'\}$.
If $|\cP'|\leq\nu(\cP')$ then it may happen that $\id (\cP')=\{0\}$, see Example \ref{ex-motivation}(1).

We use next the notions of characteristic set, generic zero and saturated ideal, which are classical in differential algebra and can be found in \cite{Ritt}, \cite{Kol}, and in the preliminaries of some more recent works as \cite{H} and \cite{GLY}.
If $|\cP'|>\nu(\cP')$, by \cite{Gao}, Lemma 3.1, $\id (\cP')$ is a differential prime ideal with generic zero $(H_{h_1}(U),\ldots ,H_{h_m} (U))$. Let $\cC$ be a characteristic set of $\id(\cP')$ (w.r.t. any ranking). The differential dimension of $\id(\cP')$ is
$\dim (\id(\cP'))=m-|\cC|\leq m-1$
and coincides with the differential transcendence degree over $\cK$ of
$\cK\langle H_{h_1}(U),\ldots ,H_{h_m} (U)\rangle$, see \cite{CH}, Section 4.2.
If $\cP$ is redundant then, there exists $\cP'\varsubsetneq \cP$, with $\nu (\cP')<|\cP'|$ and, by the previous observation,
\begin{equation}\label{eq-idp}
\{0\}\neq \id (\cP')\subset\id(\cP).
\end{equation}
Let $U'$ be the subset of $U$ such that $|U'|=\nu (\cP')$ and $\cP'\subset \bbD'\{U'\}$.
Since $\cP'$ is a set of linear differential polynomials, a characteristic set $\cA$ of $[\cP']_{\bbD'\{U'\}}$, w.r.t. the ranking $\frak{r}$, obtained for instance by \cite{H}, Algorithm 7.1, is a set of linear differential polynomials in $\bbD'\{U'\}$. If $|\cP'|>\nu(\cP')$ then by \cite{Gao}, Theorem 3.1, $\cA_0:=\cA\cap\bbD'$ is a characteristic set of $\id(\cP')$. By \cite{Kol}, Lemma 2, page 167 and the fact that the elements in $\cA_0$ are linear differential polynomials in $\bbD'$,
\[\id(\cP')=\sat(\cA_0)=[\cA_0]_{\bbD'},\]
where $\sat(\cA_0)$ is the saturated ideal of $\cA_0$ in $\bbD'$.
If $\dim (\id(\cP'))=m-1$ then $\id(\cP')=[A]_{\bbD'}$ for a linear differential polynomial $A(c_{h_1},\ldots ,c_{h_m})$ in $\bbD'$.
From the previous discussion we can conclude the following.

\begin{prop}\label{prop-idPPp}
Let $\cP'$ be a proper subsystem of $\cP$ with $\nu (\cP')<|\cP'|$. If $\dim (\id(\cP))=n-1$ then $\id(\cP)=[A]_{\bbD}$, where $A$ is a nonzero linear differential polynomial such that $\id(\cP')=[A]_{\bbD'}$.
\end{prop}

Given a nonzero linear differential polynomial $B$ in $\id (\cP)$, by \cite{R11}, Lemma 4.4 there exist unique $\cF_i\in\cK[\partial]$ such that
\begin{equation}\label{eq-linID}
B=\sum_{i=1}^{n} \cF_i (c_i)\mbox{ and }\sum_{i=1}^{n} \cF_i (H_i(U))=0.
\end{equation}
We denote a greatest common left
divisor of $\cF_{1},\ldots ,\cF_{n}$ by $\gcld (\cF_1,\ldots ,\cF_n)$. We recall \cite{R11}, Definition 4.9:
\begin{enumerate}
\item The {\sf $\id$-content} of $B$ equals $\dcont (B):=\gcld (\cF_1,\ldots ,\cF_n)$.  We say that $B$ is {\sf $\id$-primitive} if $\dcont (B)\in\cK$.

\item There exist $\cL_{i}\in\cK[\partial]$ such that $\cF_{i}=\dcont (B)\cL_{i}$,
$i=1,\ldots ,n$, and $\cL_{1},\ldots ,\cL_{n}$ are coprime. An {\sf $\id$-primitive part} of $B$ equals
\[\dprim (B):=\sum_{i=1}^n \cL_i (c_i).\]
\end{enumerate}

If $B$ belongs to $(\PS)$ then $\ord(B,c_i)\leq L_i$, $i=1,\ldots ,n$.
%Let $\sspan_{\cK} \PS$ be the linear span over $\cK$ of $\PS$.
%\[\sspan_{\cK} \ps (\cP)=\{\sum_{i=1}^n \cF_i(F_i(C,U))\mid \cF_i\in\cK [\partial], \deg(\cF_i)\leq N-o_i-\gamma\}.\]
Given a nonzero linear differential polynomial $B$ in $(\PS)$, we define the {\sf co-order with respect to $\PS$} of $B$ to be the highest positive integer $\c_{\PS} (B)$ such that $\partial^{\c_{\PS} (B)} B\in (\PS)$. Observe that, this definition was given in \cite{R11}, Definition 4.7, for a choice of $\PS$.

\begin{thm}\label{thm-dimid}
Let $\cP$ be a system of linear DPPEs as in \eqref{eq-PF}.
Let $\PS\subset\partial\cP$ and $\cU\subset\{U\}$ be sets verifying $(\ps1)$ and $(\ps2)$. If $\dim(\id(\cP))=n-1$ then $\id(\cP)=[A]_{\bbD}$, where $A$ is a linear differential polynomial verifying:
\begin{enumerate}
\item $A$ is $\id$-primitive and $A\in (\PS)\cap\bbD$.

\item $\c_{\PS}(A)=|\PS|-1-\rank(\cL(\PS,\cU))$.
\end{enumerate}
\end{thm}
\begin{proof}
We can adapt the proof of \cite{R11}, Theorem 5.2.
\end{proof}

We can also adapt the proof of \cite{RS}, Theorem 10 (1)$\Leftrightarrow$(3) to show that
\begin{equation}\label{eq-Lequiv}
\det(\cM(\PS,\cU))\neq 0\Leftrightarrow\rank(\cL(\PS,\cU))=|\PS|-1.
\end{equation}

As observed in Section \ref{sec-diff res formulas}, if $A=\det(\cM(\PS,\cU))\neq 0$ then $A$ is an element of the differential elimination ideal $\id(\cP)\cap\bbD$.
The next examples illustrate this statement in the case of the formula $\dfres(\cP)$.

\begin{exs}\label{ex-dfres}
Let $\cK=\bbQ(t)$ and $\partial=\frac{\partial}{\partial t}$.
\begin{enumerate}
\item Let us consider the system $\cP=\{F_1,F_2,F_3,F_4\}$ in $\bbD\{u_1,u_2,u_3\}$ with $\bbD=\cK\{c_1,c_2,c_3,c_4\}$ and
\begin{align*}
&F_1=c_1+5 u_{1,2}+3 u_2+u_3,\\
&F_2=c_2+u_1+u_3,\\
&F_3=c_3+u_{1,2}+u_2+u_3,\\
&F_4=c_4+u_1+u_{2,1}+u_{3,2}.
\end{align*}
Observe that $N=6$ and $\gamma(\cP)=\overline{\gamma}_2(\cP)=1$.
Thus $\dfres(\cP)$ is the determinant of the coefficient matrix of
\[\ps(\cP)=\{\partial^{3}F_1,\ldots ,F_1,\partial^5 F_2 ,\ldots ,F_2,\partial^3 F_3 ,\ldots ,F_3,\partial^3 F_4,\ldots ,F_4\}.\]
Namely
\begin{align*}
A=&\dfres(\cP)=\cL_1 (c_1)+\cL_2(c_2)+\cL_3(c_3)+\cL_4(c_4)\neq 0,\\
A=&128 c_4+192 c_3+64 \partial c_3-64 \partial c_1+128 \partial^2c_4-128 \partial^4c_2+64 \partial^2c_1\\
&-320 \partial^3c_3+64 \partial^3c_1 +256 \partial^3c_2-192 \partial^2c_3-64 c_1-128 c_2,
\end{align*}
where
\begin{align*}
&\cL_1=64 (\delta-1) (\delta+1)^2,\\
&\cL_2=-128 (\delta-1) (\delta^3-\delta^2-\delta-1),\\
&\cL_3=192+64 \delta-320 \delta^3-192 \delta^2,\\
&\cL_4=128+128 \delta^2.
\end{align*}
Thus $\cL_1,\cL_2,\cL_3$ and $\cL_4$ are coprime and $A$ is an $\id$-primitive linear polynomial in $\id(\cP)\cap\bbD$.

\item Let $x$ be a differential indeterminate over $\cK$. Set $\bbD=\cK\{x\}$ and specialize $c_1=x$, $c_2=c_3=c_4=0$ in the previous system $\cP$ to obtain $\cP^e=\{f_1,f_2,f_3,f_4\}$, which is not a system of DPPEs any more. Still $\dfres(\cP^e)=64 \partial^2x+64\partial^3x-64\partial x-64 x$, the specialization of $\dfres(\cP)$, is an element of the differential elimination ideal $[\cP^e]_{\bbD\{u_1,u_2,u_3\}}\cap\bbD$.

\item If we replace $F_1$ in $\cP$ by $c_1+ u_{1,2}+3 u_2+u_3$ then $\dfres(\cP)=0$.
\end{enumerate}
\end{exs}

If $\dfres(\cP)=0$, as in Example \ref{ex-dfres}(3), then this formula cannot be used to obtain an element of the differential elimination ideal.
The perturbation methods in \cite{R11}, Section 6 are next extended to achieve differential elimination via differential resultant formulas, even in the $\dfres(\cP)=0$ case. Although differential elimination can be always achieved via perturbations of $\dcres(\cP)$, as explained in \cite{R11}, it is worth to have similar methods available for $\dfres(\cP)$ (and other possible formulas) since $\dfres(\cP)$ is the determinant of a matrix of smaller size than the matrix used to compute $\dcres(\cP)$ in many cases ( see Example \ref{ex-dfresE}).

\para

Let $p$ be an algebraic indeterminate over $\cK$, thus $\partial (p)=0$. Denote
$\cK_p=\cK \langle p\rangle$ the differential field extension of
$\cK$ by $p$. A {\sf linear perturbation of the system } $\cP$ is a new system
\begin{equation*}\label{pDPPEs}
\cP_{\varepsilon}=\{F_i^{\varepsilon}:=F_i-p\varepsilon_i(U)\mid i=1,\ldots ,n\},
\end{equation*}
where the {\sf linear perturbation} $\varepsilon=(\varepsilon_1(U),\ldots ,\varepsilon_n(U))$ is a family of linear differential polynomials in $\cK\{U\}$.
The rest of this section is dedicated to prove that, if $\cP$ is super essential then there exists a linear perturbation $\varepsilon$ such that $\dfres(\cP_{\varepsilon})\neq 0$. For this purpose, it is shown how the proof of \cite{R11}, Theorem 6.2 applies to a more general situation than the one in \cite{R11}.

\para

Let us consider $\overline{\omega}:=(\omega_1,\ldots ,\omega_n)\in\bbN_0^n$ and $\overline{\beta}:=(\beta_1,\ldots, \beta_{n-1})\in\bbN_0^{n-1}$ verifying:
\begin{enumerate}
\item[($\beta$1)] If $\Omega:=\sum_{i=1}^n\omega_i$ and $\beta:=\sum_{j=1}^{n-1}\beta_j$ then $\Omega-\omega_i-\beta\geq 0$, $i=1,\ldots ,n$.

\item[($\beta$2)] If $\cL_{i,j}\neq 0$ then $0\leq \deg(\cL_{i,j})\leq \omega_i-\beta_j$.
\end{enumerate}
Taking assumption ($\cP$4) into consideration, for every $i$ there exists $j$ such that by $(\beta2)$
\[0\leq \Omega-\omega_i-\beta\leq \Omega-\beta_j-\beta.\]
So it can be easily verified that the next sets satisfy $(\ps1)$ and $(\ps2)$:
\begin{align*}
&\ps^{\beta,\Omega}(\cP):=\{\partial^k F_i\mid k\in [0,\Omega-\omega_i-\beta]\cap\bbZ, i=1,\ldots ,n\},\\
&\cU^{\beta,\Omega}:=\{u_{j,k}\mid k\in [0,\Omega-\beta_j-\beta]\cap\bbZ,j=1,\ldots ,n-1\}.
\end{align*}
Thus, we can define the matrix $\cM(\ps^{\beta,\Omega}(\cP),\cU^{\beta,\Omega})$, as in Section \ref{sec-diff res formulas}.

Given a nonzero differential operator  $\cL=\sum_{k\in\bbN_0}a_k\partial^k\in\bbD [\partial]$,
let $\sigma_k (\cL):=a_k$.
Let us define the {\sf $(\overline{\beta},\overline{\omega})$-symbol matrix} $\sigma_{(\overline{\beta},\overline{\omega})} (\cP)=(\sigma_{i,j})$ of $\cP$ as the $n\times (n-1)$ matrix with
\begin{equation}\label{eq-sigmaP}
\sigma_{i,j}:=\left\{
\begin{array}{ll}
\sigma_{\omega_i-\beta_j}(\cL_{i,j}), & \cL_{i,j}\neq 0,\\
0, & \cL_{i,j}= 0.
\end{array}
\right.
\end{equation}
If $\Omega\geq 1$, we can consider the coefficient matrix $\cM(\ps_h^{\beta,\Omega}(\cP))$ of
\[\ps_h^{\beta,\Omega}(\cP):=\{\partial^k H_i\mid k\in [0,\Omega-\omega_i-\beta-1]\cap\bbZ, i\in\{1,\ldots ,n\}, \Omega-\omega_i-\beta-1\geq 0\},\]
i.e. the submatrix of $\cM(\ps^{\beta,\Omega}(\cP),\cU^{\beta,\Omega})$ obtained by removing the columns indexed by $1$ and $u_{j,\Omega-\beta_j-\beta}$, $j=1,\ldots, n-1$ and the rows corresponding to the coefficients of $\partial^{\Omega-\omega_i-\beta}F_i$, $i=1,\dots ,n$.

\begin{rem}\label{rem-hatgamma}
Let $\hat{\gamma}_j$ be as in Remark \ref{rem-psuR}.
In \cite{RS} and \cite{R11}, Section 3, $\beta_j=\hat{\gamma}_j$, $j=1,\ldots ,n-1$, $\omega_i=o_i$, $i=1,\ldots ,n$ and, $\overline{\beta}\in\bbN_0^{n-1}$ and $\overline{\omega}\in\bbN_0^n$ verify ($\beta$1) and ($\beta$2).
In fact,
\begin{align*}
&\dcres (\cP)=\det(\cM(\ps^{\beta,\Omega}(\cP),\cU^{\beta,\Omega}))\mbox{ and }\\
&\dcres^h(H_1,\ldots H_n)=\det(\cM(\ps_h^{\beta,\Omega}(\cP))).
\end{align*}
\end{rem}

Using the matrix $\sigma_{(\overline{\beta},\overline{\omega})} (\cP)$, we can adapt the proof of \cite{RS}, Theorem 10 (1)$\Leftrightarrow$(2) to show that
\begin{equation}\label{eq-equiv2}
\det(\cM(\ps^{\beta,\Omega}(\cP),\cU^{\beta,\Omega}))\neq 0\Leftrightarrow\det(\cM(\ps_h^{\beta,\Omega}(\cP)))\neq 0.
\end{equation}

Let us assume, in addition, that $\overline{\omega}\in\bbN_0^n$ and $\overline{\beta}\in\bbN_0^{n-1}$ verify
\begin{enumerate}
\item[($\beta$3)] $\omega_i-\beta_{n-i}\geq 0$, $i=1,\ldots ,n-1$,
\end{enumerate}
to define the linear perturbation $\phi=(\phi_1(U),\ldots ,\phi_n(U))$ by
\begin{equation}\label{eq-phi}
\phi_i (U)=\left\{\begin{array}{ll}
u_{n-1,\omega_1-\beta_{n-1}},& i=1,\\
u_{n-i,\omega_i-\beta_{n-i}}+ u_{n-i+1},& i=2,\ldots ,n-1,\\
u_{1},& i=n.
\end{array}\right.
\end{equation}
Let $\sigma^n_{(\overline{\beta},\overline{\omega})}(\cP)$ be the $(n-1)\times (n-1)$ matrix obtained by removing the $n$th row of $\sigma_{(\overline{\beta},\overline{\omega})}(\cP)$.

\begin{lem}\label{lem-phi}
Let us consider $\overline{\omega}\in\bbN_0^n$ and $\overline{\beta}\in\bbN_0^{n-1}$ verifying ($\beta$1), ($\beta$2), ($\beta$3) and $\omega_n\geq \omega_{n-1}\geq \cdots \geq \omega_1$.
Given the linear perturbation $\phi$ defined by \eqref{eq-phi} it holds that $\det(\sigma^n_{(\overline{\beta},\overline{\omega})}(\cP))\neq 0$ and
$\det(\cM(\ps_h^{\beta,\Omega}(\cP_{\phi})))\neq 0$.
\end{lem}
\begin{proof}
The proof of \cite{R11}, Proposition 6.1 holds under assumptions ($\beta$1), ($\beta$2), ($\beta$3) and $\omega_n\geq \omega_{n-1}\geq \cdots \geq \omega_1$, it does not make use of the precise definition of $\beta_j$ or $\omega_i$ in \cite{R11} (see Remark \ref{rem-hatgamma}). Thus we can adapt the proof of \cite{R11}, Proposition 6.1(2) to conclude that $\det(\cM(\ps_h^{\beta,\Omega}(\cP_{\phi})))\neq 0$.
\end{proof}

From \eqref{eq-equiv2} and Lemma \ref{lem-phi} the next result follows. Observe that \cite{R11}, Theorem 6.2 coincides with the next proposition for $\overline{\beta}$ and $\overline{\omega}$ as in Remark \ref{rem-hatgamma}.

\begin{prop}
Let us consider $\overline{\omega}\in\bbN_0^n$ and $\overline{\beta}\in\bbN_0^{n-1}$ verifying ($\beta$1), ($\beta$2), ($\beta$3) and $\omega_n\geq \omega_{n-1}\geq \cdots \geq \omega_1$. Given a linear system $\cP$ as in \eqref{eq-PF}, there exists a linear perturbation $\phi$ such that
\[\det(\cM(\ps^{\beta,\Omega}(\cP_{\phi}),\cU^{\beta,\Omega}))\neq 0.\]
\end{prop}

We denote $\gamma_j(\cP)$, $\overline{\gamma}_j(\cP)$ and $\underline{\gamma}_j(\cP)$ simply by $\gamma_j$, $\overline{\gamma}_j$ and $\underline{\gamma}_j$ in the remaining parts of this section.

\begin{rem}\label{rem-cQ}
There exists a system $\cQ=\{G_i\mid i=1,\ldots ,n\}$ in $\bbD\{U\}$
such that
\[G_i(u_{1,\underline{\gamma}_1},\ldots ,u_{n-1,\underline{\gamma}_{n-1}})=F_i(u_1,\ldots ,u_{n-1}).\]
Observe that $\ord(G_i)\leq o_i$, $i=1,\ldots ,n$ and $\gamma_j (\cQ)$ may be different from $\gamma_j$.
Let $\overline{\omega}=(o_1,\ldots ,o_n)$, $\overline{\beta}=(\gamma_1,\ldots,\gamma_{n-1})$ and observe that
$\cM(\ps^{\beta,\Omega}(\cQ),\cU^{\beta,\Omega})$ may be different from $\cM(\cQ)$ (as in Section \ref{sec-diff res formulas})
but $\cM(\ps^{\beta,\Omega}(\cQ),\cU^{\beta,\Omega})$
is obtained by reorganizing the columns of $\cM(\cP)$, thus
\[\dfres (\cP)=\pm \det (\cM(\ps^{\beta,\Omega}(\cQ),\cU^{\beta,\Omega})).\]
\end{rem}

We may assume w.l.o.g. that $o_n\geq o_{n-1}\geq \cdots \geq o_1$ and otherwise rename the polynomials in $\cP$.
Let us assume that $\cP$ is super essential and let
$\mu:=\mu_n$ be as in \eqref{eq-mui}, thus
\[o_i-\overline{\gamma}_{\mu(i)}\geq 0, i=1,\ldots ,n-1.\]
We define the linear perturbation $\varepsilon=(\varepsilon_1(U),\ldots ,\varepsilon_n(U))$ by
\begin{equation}\label{eq-varepsilon}
\varepsilon_i (U)=\left\{\begin{array}{ll}
u_{\mu(1),o_1-\overline{\gamma}_{\mu(1)}},& i=1,\\
u_{\mu(i),o_i-\overline{\gamma}_{\mu(i)}}+ u_{\mu(i-1),\underline{\gamma}_{\mu(i-1)}},& i=2,\ldots ,n-1,\\
u_{\mu(n-1),\underline{\gamma}_{\mu(n-1)}},& i=n.
\end{array}\right.
\end{equation}
%Let $\overline{\gamma}=(\overline{\gamma}_1,\ldots,\overline{\gamma}_{n-1})$ and observe that for the previous perturbation %$\rank(\sigma_{\overline{\gamma}}(\cP_{\varepsilon}))=n-1$.
We use the perturbation given by \eqref{eq-varepsilon} to prove the next result, but there are other perturbations that may serve the same purpose.

\begin{thm}\label{th-dres(Fphi)}
Given a super essential system $\cP$ as in \eqref{eq-PF}, there exists a linear perturbation $\varepsilon$ such that
the differential resultant $\dfres (\cP_{\varepsilon})$ is a nonzero polynomial in $\cK [p]\{C\}$.
\end{thm}
\begin{proof}
Let $\varepsilon$ be the linear perturbation defined by \eqref{eq-varepsilon}.
Let $\cQ=\{G_i\mid i=1,\ldots ,n\}$ be as in Remark \ref{rem-cQ}. If $\psi=(\psi_1(U),\ldots ,\psi_n(U))$ is the linear perturbation defined by
\begin{equation}\label{eq-phi2}
\psi_i (U)=\left\{\begin{array}{ll}
u_{\mu(1),o_1-\gamma_{\mu(1)}},& i=1,\\
u_{\mu(i),o_i-\gamma_{\mu(i)}}+ u_{\mu(i-1)},& i=2,\ldots ,n-1,\\
 u_{\mu(n-1)},& i=n
\end{array}\right.
\end{equation}
then $G_i^{\psi}(u_{1,\underline{\gamma}_1},\ldots ,u_{n-1,\underline{\gamma}_{n-1}})=F_i^{\varepsilon}(u_1,\ldots ,u_{n-1})$.
For $\overline{\beta}=(\gamma_1,\ldots ,\gamma_{n-1})$ and $\overline{\omega}=(o_1,\ldots ,o_n)$, by Remark \ref{rem-cQ}
\[\dfres (\cP_{\varepsilon})=\pm \det (\cM(\ps^{\beta,\Omega}(\cQ_{\psi}),\cU^{\beta,\Omega})).\]
We can assume w.l.o.g. that $o_n\geq o_{n-1}\geq \cdots \geq o_1$ and that $\mu=(n-1,\ldots ,1)$ (otherwise rename the elements in $U$), thus $\psi=\phi$, the perturbation given by \eqref{eq-phi}. By Lemma \ref{lem-phi} and \eqref{eq-equiv2} $\det(\cM(\ps^{\beta,\Omega}(\cQ_{\phi}),\cU^{\beta,\Omega}))\neq 0$ and the result is proved.
\end{proof}

By Theorem \ref{th-dres(Fphi)}, if $\cP$ is super essential then there exists a linear perturbation $\varepsilon$ such that $\dfres (\cP_{\varepsilon})$ is a nonzero differential polynomial in $\cK[p]\{C\}$.
Let $D_{\varepsilon}$ be the lowest degree of $\dfres (\cP_{\varepsilon})$ in $p$ and $A_{D_{\varepsilon}}$ the coefficient of $p^{D_{\varepsilon}}$ in $\dfres (\cP_{\varepsilon})$. It can be proved as in \cite{R11}, Lemma 6.4 that $A_{D_{\varepsilon}}\in (\ps(\cP))\cap\bbD$. Let $A_{\varepsilon}(\cP)$ be an $\id$-primitive part of $A_{D_{\varepsilon}}$. Thus the polynomials $A_{D_{\varepsilon}}$ and $A_{\varepsilon}(\cP)$ can both be used for differential elimination of the variables $U$ from the system $\cP$.

\begin{lem}\label{lem-AE}
$A_{\varepsilon}(\cP)$ is a nonzero linear $\id$-primitive differential polynomial in  $(\ps(\cP))\cap\bbD$.
\end{lem}

The next example illustrates these perturbation methods. Computations were performed with Maple 15 using our implementation of $\dfres(\cP)$. Making this implementation user friendly and publicly available is left as a future project.

\begin{ex}\label{ex-dfresE}
Let $\cP$ be the system of Example \ref{ex-dfres}(3). The perturbed system $\cP_{\varepsilon}$ with $\varepsilon$ as in \eqref{eq-varepsilon} and $\mu=\mu_4=(3,1,2)$ is
\[\cP_{\varepsilon}=\{F_1-p u_{3,2},F_2-p(u_1+u_3), F_3-p(u_{2,1}+u_1),F_4-p u_2\}.\]
By Theorem  \ref{th-dres(Fphi)} we have $\dfres(\cP_{\varepsilon})\neq 0$, in fact
\[\dfres(\cP_{\varepsilon})=p^2 P(c_1,c_2,c_3,c_4),\mbox{ with }P=A_{D_{\varepsilon}}+p A'\in\cK[p]\{C\}\]
and $A_{D_{\varepsilon}}=\cF_1(c_1)+\cF_2(c_2)+\cF_3(c_3)+\cF_4(c_4)$, with
\begin{align*}
&\cF_1=-24(\delta-1)(\delta+1)^2=\cL (-\delta-1),\\
&\cF_2= -48 (\delta-1)(\delta+1)(\delta^2+1)=\cL (-2\delta^2-2),\\
&\cF_3=24(\delta-1)(\delta+3)(\delta+1)=\cL (\delta+3),\\
&\cF_4=48(\delta-1)(\delta+1)=2\cL ,
\end{align*}
and $\cL=24(\delta-1)(\delta+1)$. Finally
\begin{equation}\label{eq-AE}
A_{\varepsilon}(\cP)=-\partial c_1-c_1-2\partial^2c_2-2c_2+\partial c_3+3c_3+2c_4.
\end{equation}
Observe that $\dfres(\cP_{\varepsilon})$ is the determinant of a matrix of size $18\times 18$. If we use $\dcres(\cP_{\phi})$, with $\phi$ as in \cite{R11}(4), to eliminate the differential variables $u_1,u_2,u_3,u_4$, the polynomial \eqref{eq-AE} is also obtained but $\dcres(\cP_{\phi})$ is the determinant of a matrix of size $22\times 22$.
\end{ex}

%%%%%%%%%%%%%%%%%%%%%%%%%%%%%%%%%%%%%%%%%%%%%%%%%%%%%%%%%%%%%%%%%%%
\section{Sparse linear differential resultant}\label{sec-Computation of the sparse}
%%%%%%%%%%%%%%%%%%%%%%%%%%%%%%%%%%%%%%%%%%%%%%%%%%%%%%%%%%%%%%%%%%%

The field $\bbQ$ of rational numbers is a field of constants of the derivation $\partial$.
For $i=1,\ldots ,n$ and $j=1,\ldots ,n-1$, let us consider subsets $\frak{S}_{i,j}$ of $\bbN_0$
to be the supports of differential operators
\[\cG_{i,j}:=\left\{\begin{array}{lc}
\sum_{k\in\frak{S}_{i,j}} c_{i,j,k}\partial^k & \frak{S}_{i,j}\neq\emptyset,\\
0 & \frak{S}_{i,j}=\emptyset,
\end{array}
\right.\]
whose coefficients are differential indeterminates over $\bbQ$ in the set
\[\overline{C}:=\cup_{i=1}^n\cup_{j=1}^{n-1}\{c_{i,j,k}\mid k\in\frak{S}_{i,j}\}.\]

Let $\bbF_i$, $i=1,\ldots ,n$ be a {\sf generic sparse linear differential polynomial} as follows,
\begin{equation}\label{eq-bbF}
\bbF_i:=c_i+\sum_{j=1}^{n-1}\cG_{i,j}(u_j)=c_i+\sum_{j=1}^{n-1}\sum_{k\in \frak{S}_{i,j}}c_{i,j,k}u_{j,k}.
\end{equation}
In this section, $\cK=\bbQ\langle \overline{C}\rangle$, a differential field extension of $\bbQ$ with derivation $\partial$, and $\bbD=\cK\{C\}$. Consider the system of linear DPPEs in $\bbD\{U\}$
\[\frak{P}:=\{\bbF_i=c_i-\bbH_i(U)\mid i=1,\ldots ,n\}.\]
Let us assume that the order of $\bbF_i$ is $o_i\geq 0$, $i=1,\ldots ,n$ so that, if $\cG_{i,j}\neq 0$,
\[\frak{S}_{i,j}\subset I_{i,j}(\frak{P})=[\underline{\gamma}_j(\frak{P}),o_i-\overline{\gamma}_j(\frak{P})]\cap\bbZ.\]
We also assume $(\cP 4)$ and observe that $(\cP 2)$ and $(\cP 3)$, in Section \ref{sec-Preliminary}, are verified.

By \cite{LGY}, Corollary 3.4, the dimension of $\id (\frak{P})=[\frak{P}]_{\bbD\{U\}}\cap\bbD$ is $n-1$ if and only if $\frak{P}$ is a differentially essential system.
In such case, $\id (\frak{P})=\sat (R)$, the saturation ideal of a unique (up to scaling) irreducible differential polynomial $R(c_1,\ldots ,c_n)$ in $\bbD=\cK\{C\}$. By clearing denominators when necessary, we can assume that $R\in\bbQ\{\overline{C},C\}$. By \cite{LGY}, Definition 3.5, $R$ is the {\rm sparse differential resultant} of $\frak{P}$. We will denote it by $\dres (\frak{P})$ and call it the {\sf sparse linear differential resultant} of $\frak{P}$.

\begin{rem}\label{rem-dres}
Given a differentially essential system $\frak{P}$, by Theorem \ref{thm-dimid}, $\id(\frak{P})=[\dres(\frak{P})]_{\bbD}$ and $\dres(\frak{P})$ is a linear $\id$-primitive differential polynomial in $\id (\frak{P})$.
Observe that $\dres (\frak{P})$ is the implicit equation of the system of linear DPPEs $\frak{P}$, as defined in \cite{RS}, Definition 2.
Furthermore, given $\PS\subset\partial\frak{P}$ and $\cU\subset\{U\}$ verifying $(\ps1)$ and $(\ps2)$, it holds that:
\begin{enumerate}
\item $\dres (\frak{P})$ belongs to $(\PS)\cap\bbD$ and,

\item $\c_{\PS} (\dres (\frak{P}))=|\PS|-1-\rank (\cL(\PS,\cU))$.
\end{enumerate}
\end{rem}

\begin{prop}\label{prop-dmac}
Let $\frak{P}$ be a differentially essential system. Given $\PS\subset \partial\frak{P}$ and $\cU\subset\{U\}$ verifying $(\ps1)$ and $(\ps2)$, the following statements are equivalent:
\begin{enumerate}
\item $\det (\cM(\PS,\cU))\neq 0$.

\item $\ord(\dres (\frak{P}),c_i)\leq L_i$, $i=1,\ldots,n$ and there exists $k\in\{1,\ldots ,n\}$ such that $\ord(\dres (\frak{P}),c_k)= L_k$.
\end{enumerate}
Furthermore, if $\det (\cM(\PS,\cU))\neq 0$ then $\det (\cM(\PS,\cU))=\alpha \dres (\frak{P})$ for some nonzero $\alpha\in\cK$.
\end{prop}
\begin{proof}
By \eqref{eq-Lequiv}, 1 is equivalent to $\rank(\cL(\PS,\cU))=|\PS|-1$.
Furthermore, by Remark \ref{rem-dres}(2), it is equivalent to $\c_{\PS}(\dres (\frak{P}))=0$ and, since $\dres (\frak{P})\in (\PS)$, this is equivalent to 2. Finally, if $D=\det (\cM(\PS,\cU))\neq 0$ then $D\in (\PS)\cap\bbD$ and $\c_{\PS}(D)=0$ as well. Since $\dres (\frak{P})$ is $\id$-primitive, there exists a nonzero $\alpha\in\cK$ such that $D=\alpha \dres (\frak{P})$.
\end{proof}

If $\frak{P}$ is differentially essential then there exists a unique super essential subsystem $\frak{P}^*$ of $\frak{P}$, by Theorem \ref{thm-seunique}. If $\frak{P}$ is super essential then $\frak{P}^*=\frak{P}$, otherwise, by Theorem \ref{thm-subSC}, $\frak{P}^*$ can be obtained by \eqref{eq-P*}.

\begin{lem}\label{lem-idPp0}
Given $i\in\{1,\ldots ,n\}$, if $\det (X_i(\frak{P}))\neq 0$ then, for every subset $\frak{P}'$ of $\frak{P}_i$, the differential ideal $\id (\frak{P}')$ contains no  linear differential polynomial.
\end{lem}
\begin{proof}
Let $\frak{P}'=\{\bbF_{h_1},\ldots ,\bbF_{h_m}\}$, with $h_1,\ldots,h_m\in\{1,\ldots ,n\}\backslash\{i\}$. By Lemma \ref{lem-Xi}, there exists $\mu_i$ as in \eqref{eq-mui}, such that $\cG_{h_t,\mu_i(h_t)}\neq 0$, $t=1,\ldots ,m$.
%Let $\id (\frak{P}')$ be as in \eqref{eq-idPp}.
Given a linear differential polynomial $B\in\id (\frak{P}')$, by \eqref{eq-linID}, there exist $\cF_{h_1},\ldots ,\cF_{h_m}\in\cK[\partial]$ such that
$\sum_{t=1}^m \cF_{h_t}(\bbH_{h_t}(U))=0$, ($B=\sum_{t=1}^m \cF_{h_t}(c_{h_t})$).
Replacing by zero the coefficients of $\cG_{h_t,j}$, for $t=1,\ldots ,m$ and $j\neq \mu_i(h_t)$, this would contradict that  $u_{\mu_i(h_1)},\ldots ,u_{\mu_i(h_m)}$ are differentially independent. This proves that $B$ does not  exist.
\end{proof}

From the previous lemma, we conclude that to compute $\dres(\frak{P})$ we need to use all the elements of $\frak{P}^*$.

\begin{cor}
Let $\frak{P}$ be a differentially essential system with super essential subsystem $\frak{P}^*$. The following statements hold:
\begin{enumerate}
\item $\dres (\frak{P})$ is a linear $\id$-primitive differential polynomial in $\id (\frak{P}^*)$.

\item $\frak{P}^*$ is the smallest subset of $\frak{P}$ such that $\dres(\frak{P})\in\id (\frak{P}^*)$.

\item If $\dfres (\frak{P}^*)\neq 0$ then $\dres (\frak{P})=\frac{1}{\alpha}\dfres (\frak{P}^*)$, for a nonzero differential polynomial $\alpha$ in $\cK$.
\end{enumerate}
\end{cor}
\begin{proof}
By Proposition \ref{prop-idPPp}, $\dres (\frak{P})=\dres (\frak{P}^*)$, which proves $1$. Statement $2$ is a consequence of Lemma \ref{lem-idPp0} and statement $3$ follows from Proposition \ref{prop-dmac}.
\end{proof}

Observe that the extraneous factor $\alpha$ does not depend on the variables in $C=\{c_1,\ldots ,c_n\}$, since $\alpha$ is a nonzero differential polynomial in $\cK=\bbQ\langle \overline{C}\rangle$.

\begin{ex}
Let us consider the following system $\frak{P}$ in $\bbD\{u_1,u_2\}$
\begin{align*}
&\bbF_1=c_1+c_{1,1,0}u_1+c_{1,2,1} u_{2,1},\\
&\bbF_2=c_2+c_{2,1,2} u_{1,2},\\
&\bbF_3=c_3+c_{3,1,0} u_1+c_{3,2,1} u_{2,1}.
\end{align*}
The matrix $X(\frak{P})$ is as in Example \ref{exs-se-de}(1), thus $\frak{P}$ is super essential. The formula $\dfres(\frak{P})$ is the determinant of the matrix $\cM(\frak{P})$ whose rows can be reorganized to get
\[
 \left[
\begin{array}{cccccccc}
c_{1,2,1} & 0 & 2\,\partial c_{1,2,1} & c_{1,1
,0} & \partial^2 c_{1,2,1} & 2\,\partial c_{1,1,0} & \partial^2 c_{1,1,0} & \partial^2 c_{1} \\
0 & c_{1,1,2} & 0 & \partial c_{1,1,2} & 0 & 0 & 0
 & \partial c_{2} \\
c_{3,2,1} & 0 & 2\,\partial c_{3,2,1} & c_{3,1
,0} & \partial^2 c_{3,2,1} & 2\,\partial c_{3,1,0} & \partial^2 c_{3,1,0} & \partial^2 c_{3} \\
0 & 0 & c_{1,2,1} & 0 & \partial c_{1,2,1} & c_{1
,1,0} & \partial c_{1,1,0} & \partial c_{1} \\
0 & 0 & 0 & c_{1,1,2} & 0 & 0 & 0 & c_{2} \\
0 & 0 & c_{3,2,1} & 0 & \partial c_{3,2,1} & c_{3,1,0} & \partial c_{3,1,0} & \partial c_{3} \\
0 & 0 & 0 & 0 & c_{1,2,1} & 0 & c_{1,1,0} & c_{1} \\
0 & 0 & 0 & 0 & c_{3,2,1} & 0 & c_{3,1,0} & c_{3}
\end{array}
 \right].
\]
Namely $\dfres(\frak{P})=-c_{1,1,2}\dres(\frak{P})$, with
\[\ord(\dres(\frak{P}),c_1)=2, \ord(\dres(\frak{P}),c_2)=0\mbox{ and }\ord(\dres(\frak{P}),c_3)=2.\]
\end{ex}

Let us denote by $\sigma (\frak{P})$ the $(\overline{\beta},\overline{\omega})$-symbol matrix of $\frak{P}$ as in \eqref{eq-sigmaP}, for $\overline{\beta}=(\overline{\gamma}_1(\frak{P}),\ldots ,\overline{\gamma}_{n-1}(\frak{P}))$ and $\overline{\omega}=(o_1,\dots ,o_n)$.

\begin{lem}
If $\rank(\sigma (\frak{P}))<n-1$ then $\dfres (\frak{P})=0$.
\end{lem}
\begin{proof}
Observe that the submatrix of $\cM (\frak{P})$ whose columns are indexed by $u_{j,N-\overline{\gamma}_j(\frak{P})-\gamma(\frak{P})}$, $j=1,\ldots ,n-1$ has as nonzero rows the rows of $\sigma (\frak{P})$.
\end{proof}
Examples of systems $\frak{P}$ with $\rank(\sigma (\frak{P}))=n-1$ such that $\dfres (\frak{P})=0$ have not been found so far, but their non existence has not be proved either.
In addition observe that, even if  $\dfres(\frak{P})\neq 0$, for an specialization $\cP$ of $\frak{P}$ it may happen that $\dfres(\cP)= 0$, see Example \ref{ex-fin}.

\begin{ex}\label{ex-fin}
Given the differentially essential system of generic differential polynomials $\frak{P}=\{\bbF_1,\bbF_2,\bbF_3,\bbF_4\}$,
\begin{align*}
&\bbF_1=c_1+c_{1,1,0} u_1+c_{1,1,1} u_{1,1}+c_{1,3,0} u_3+c_{1,3,1} u_{3,1},\\
&\bbF_2= c_2+c_{2,2,0} u_2+c_{2,2,1} u_{2,1},\\
&\bbF_3= c_3+c_{3,1,0} u_1+c_{3,3,0} u_3,\\
&\bbF_4= c_4+c_{4,1,0} u_1+c_{4,2,0} u_2+c_{4,3,0} u_3,
\end{align*}
let us consider the specialization $\cP$ of $\frak{P}$
\[
\cP=\{ c_1+ u_1+ 2 u_{1,1}+ u_3+ 2 u_{3,1}, c_2+u_2+u_{2,1}, c_3+u_1+ u_3, c_4+ u_1+u_2+u_3\}.
\]
It holds that $\dfres(\frak{P})\neq 0$ but $\dfres(\cP)= 0$, even thought $\cP$ is super essential and $\cM(\cP)$ has no zero columns. We can check, applying \cite{R11}, Algorithm 7.1,
 that $\dim \id (\cP)<3$.
\end{ex}

\para

By Theorem \ref{th-dres(Fphi)}, if $\frak{P}$ is super essential then there exists a linear perturbation $\varepsilon$ such that $\dfres (\frak{P}_{\varepsilon})$ is a nonzero differential polynomial in $\cK[p]\{C\}$.
Let $A_{\varepsilon}(\frak{P})$ be as in Section \ref{sec-linDPPEs}. Since $A_{\varepsilon}(\frak{P})$ is a nonzero linear $\id$-primitive differential polynomial in  $(\ps(\frak{P}))\cap\bbD$, by Lemma \ref{lem-AE}, the next result follows.

\begin{prop}\label{prop-Aphi}
Given a super essential system $\frak{P}$, there exists a linear perturbation $\varepsilon$ and a nonzero $\alpha\in\cK$ such that $\alpha\dres(\frak{P})= A_{\varepsilon}(\frak{P})$.
\end{prop}

\para

From Proposition \ref{prop-Aphi} we can derive bounds for the order of $\dres (\frak{P})$ in the variables $c_1,\ldots ,c_n$.
Let us consider a differentially essential system $\frak{P}$, of generic sparse linear differential polynomials, and the super essential system $\frak{P}^*$ of $\frak{P}$. If $I^*:=\{i\in\{1,\ldots ,n\}\mid \bbF_i\in\frak{P}^*\}$ and $N^*:=\sum_{i\in I^*} o_i$ then, for $i=1,\ldots ,n$,
\begin{equation}\label{eq-bounds}
\begin{array}{l}
\ord(\dres(\frak{P}),c_i)=-1\mbox{ if }i\notin I^*,\\
\ord(\dres(\frak{P}),c_i)\leq N^*-o_i-\gamma(\frak{P}^*)\mbox{ if }i\in I^*.
\end{array}
\end{equation}

It was proved in \cite{R11} that $\dcres(\frak{P}_{\phi})\neq 0$, for a linear perturbation $\phi$, see Remark \ref{rem-psuR2}. Thus, given a differentially essential system $\frak{P}$ and reasoning as above $\ord (\dres(\frak{P}))\leq N-o_i-\hat{\gamma}$, $i=1,\ldots ,n$. Observe that \eqref{eq-bounds} is an improvement of $N-o_i-\hat{\gamma}$ whenever $\frak{P}$ is not super essential or if $\gamma (\frak{P})> \hat{\gamma}$.

It cannot be said that \eqref{eq-bounds} are the best bounds in the linear case, since an improvement (in some cases) has been just presented in \cite{LYG}, 5.2. But at least \eqref{eq-bounds} are the best bounds obtained so far from differential resultant formulas.

%%%%%%%%%%%%%%%%%%%%%%%%%%%%%%%%%%%%%%%%%%%%%%%%%%%%%%%%%%%%%%%%%%%
\section{Concluding remarks}\label{sec-Conclusion}
%%%%%%%%%%%%%%%%%%%%%%%%%%%%%%%%%%%%%%%%%%%%%%%%%%%%%%%%%%%%%%%%%%%

In this paper, a global approach to differential resultant formulas was provided for systems of linear nonhomogeneous ordinary differential polynomials $\cP$. In particular, the formula $\dfres(\cP)$ was defined, which is an improvement of the existing formulas since, it is the first one given as the determinant of a matrix $\cM(\cP)$ with nonzero columns, for $\cP$ super essential. In addition, every system $\cP$ was proved to have a super essential subsystem $\cP^*$ and therefore the formula $\dfres(\cP^*)$ can be computed in all cases.

Still $\dfres(\cP^*)$ may be zero and for this reason, given a system $\cP$ of linear DPPEs, the existence of a linear perturbation $\varepsilon$ such that $\dfres(\cP^*_{\varepsilon})\neq 0$ was proved. Therefore, to achieve differential elimination of the variables $U$ of the system $\cP$ the polynomial $A_{\varepsilon}(\cP^*)$ can always be used, which is a nonzero linear $\id$-primitive differential polynomial in $\id(\cP)\cap\bbD$, as in Section \ref{sec-linDPPEs}.

Certainly, there is room for improvement, regarding differential resultant formulas in the linear case.
For a generic system $\frak{P}$ of sparse linear nonhomogeneous ordinary differential polynomials, $\dfres(\frak{P}^*)$ may be zero in some cases. Finding a differential resultant formula for $\frak{P}^*$ that is nonzero in all cases would improve further the existing ones. It would be given by the determinant of a matrix of smaller size (in some cases) and the candidate to be the numerator of a Macaulay style formula for $\dres(\frak{P}^*)$. Nevertheless, even in such ideal situation perturbation methods will be needed to use such formulas for applications, namely to perform differential elimination with specializations $\cP$ of $\frak{P}$, see Example \ref{ex-fin}.

\vspace{0.5cm}{\bf \noindent Acknowledgements.}
I started this work during a research visit at the Institute of Mathematics, Goethe Universit\"at and I am grateful to T. Theobald and his group for providing a very pleasant work atmosphere.
I would like to thank X.S. Gao, W. Li and C.M. Yuan for their helpful comments.
 
This work was developed, and partially supported by the "Ministerio de Ciencia e Innovaci\'{o}n"  under the project MTM2008-04699-C03-01 and by the "Ministerio de Econom\'\i a y Competitividad" under the project MTM2011-25816-C02-01. I am a member of the Research Group ASYNACS (Ref. CCEE2011/R34).

\end{document}